\documentclass[10pt]{article}
\usepackage{latexsym}
\usepackage{amsmath,amsthm, amssymb}
\usepackage{amssymb,amsfonts}
\usepackage{color}
\usepackage{graphicx}
\usepackage{epsfig}
\usepackage{mathrsfs}
\usepackage{empheq}
\definecolor{purple}{rgb}{0.65, 0, 1}
\definecolor{orange}{rgb}{1,.5,0}

\def\II{{\rm I\kern-0.5exI}}
\def\III{{\rm I\kern-0.5exI\kern-0.5exI}}
\numberwithin{equation}{section}
\newtheorem{theorem}{Theorem}[section]
\newtheorem{lemma}[theorem]{Lemma}
\newtheorem{definition}[theorem]{Definition}
\newtheorem{proposition}[theorem]{Proposition}
\newtheorem{corollary}[theorem]{Corollary}

\theoremstyle{definition}
\newtheorem{remark}[theorem]{Remark}

\begin{document}
\title{An aggregation equation with degenerate diffusion: qualitative property of solutions}
\author{Lincoln Chayes, Inwon Kim and Yao Yao}
\maketitle
\begin{abstract}
\noindent We study a nonlocal aggregation equation with degenerate diffusion, set in a periodic domain. 
This equation represents the generalization to 
$m > 1$ of the McKean--Vlasov equation where here the ``diffusive'' portion of the dynamics are governed by \textit{Porous medium} self--interactions.  We focus primarily on $m\in(1,2]$ with particular emphasis on 
$m = 2$.  
In general, we establish regularity properties and, for small interaction, exponential decay to the uniform stationary solution.  For $m=2$, we obtain essentially sharp results on the rate of decay for the entire
regime up to the (sharp) transitional value of the interaction parameter.

\end{abstract}

\section{Introduction}
In this paper we study weak solutions of the equation:
\begin{equation}
\label{pde}
\rho_t = \Delta (\rho^m) + \theta L^{d(2-m)} \nabla\cdot(\rho\nabla (V * \rho))  \quad \text{in } \mathbb{T}_L^d \times [0,\infty),
\end{equation}

\noindent
where $*$ stands for convolution, and the space domain is the $d$-dimension torus with scale $L$, defined as $\mathbb{T}^d_L :=[-\frac{L}{2}, \frac{L}{2}]^d$ with periodic boundary condition. We assume that $V$ smooth 
and integrable  (for precise conditions, see \textbf{(V1)}-\textbf{(V2)}  in Section \ref{app_to_nonlocal}), and that $\theta$ is a positive constant.  The primary focus of this work concerns the cases
$m \in (1,2]$ -- especially $m = 2$.
In addition, we remark
that a goal of interest (not always achieved) is to acquire results uniform in
$L$ for $L\gg 1$.  
%Certainly there is no particular interest in the opposite extreme and so, on occasion, we will assume 
%$L \gtrsim 1$ in order to alleviate unnecessary labour.
We point out that, in the absence of the aggregation term (i.e., when $V=0$) our equation becomes the well-known {\it Porous medium equation} (PME): 
$$
\rho_t - \Delta(\rho^m)=0.
$$

Note that, formally (and in actuality)  the mass of the solution 
to Eq.(\ref{pde})
is preserved over time:
$$
\int \rho(x,0) dx = \int \rho(x,t) dx\hbox{ for all } t>0.
$$
Without loss of generality, we can thus assume $\int \rho(x,0) dx = 1$ and 
results for other normalizations can be obtained by scaling.
%\noindent \textcolor{purple}{Should, perhaps note that \textit{if} $m_{0}$ is of order unity independent of $L$, than any difference in physics/mathematics between $m_{0} = 1$ and $m_{0}\neq 1$ can be absorbed into the definition of $\theta$.  This may prove to be of slight importance since we are identifying the transition point in the case $m = 2$ and do not want to adjust our formulae for spurious cases of $m_{0}\neq 1$.  Also:  $m_{0}$ not particularly good notation in present context.}

In the context of biological aggregation, $\rho$ represents the population density which locally disperses by the diffusion term, while $V$ is the sensing (interaction) kernel that models the long-range attraction;
Eq.\eqref{pde} is relevant for models which have been introduced by \cite{bcm} and \cite{tbl}, and further studied by \cite{bcm2} and \cite{bf}.   The above equation can also be regarded as the evolution equation for a strongly interacting fluid:  The $V$ represents the long distance component of the interaction while short distance interactions -- and entropic effects -- are accounted for by the \textit{degeneracy} ($m > 1$) in the diffusion term.
Mathematically, the equation exhibits an interesting competition between degenerate diffusion and nonlocal aggregation.   

When $V$ satisfies $V(x)=V(-x)$,
 Eq.\eqref{pde} is a gradient flow of the following energy with respect to the Wasserstein metric:
\begin{equation}
\label{energy}
\mathcal{F}_\theta(\rho): = \int_{\mathbb{T}_L^d}\frac{1}{m-1} (\rho^m - \rho)
+ \frac{1}{2}\theta L^{d(2-m)} \rho (V * \rho) dx.
\end{equation}

\noindent
Note that as $m\to 1$, the first term in the integrand of $\mathcal{F}_\theta$ converges to $\rho \log \rho$ which we refer to as the $m = 1$ case.
Using above energy structure, the existence and uniqueness properties of Eq.\eqref{pde}, in some appropriate Sobolev space, has been obtained in \cite{bs} (also see \cite{s} and \cite{brb} for relevant results).

Compared to the well-posedness theory based on energy methods, few results has been known for pointwise behaviors of solutions,  due to the lack of regularity estimates: the difficulty for regularity analysis lie mainly  in the fact that the solutions are not necessarily positive 
(i.e., {\it strictly} positive)
due to the degenerate diffusion. 
%\noindent \textcolor{purple}{(3)  Can we site a reference for this; or is this just general wisdom.  Methinks this is positivity preserving.  Or, do you mean \textit{strictly} positive?}
This is what we address in the first part of our paper.  In addition, in the non--compact setting, the plausible limiting solutions tend to be trivial; here, since mass is conserved, even in the ``worst'' of cases, there is always the uniform stationary state.  Most of the rest of this work is concerned with the approach to the asymptotic state.

\noindent $\bullet$
{\it Regularity properties}
\hspace{.125 cm}
Due to the degenerate diffusion,  one cannot expect smooth solutions of Eq.\eqref{pde}: even for (PME), H\"{o}lder regularity is optimal, as verified by the self-similar (Barenblatt) solutions (see \cite{v}). 
On the other hand, the solution of (PME) is indeed H\"{o}lder continuous (again, see \cite{v}), which motivates the question of H\"{o}lder regularity of the solution of our problem Eq.\eqref{pde}.

Note that, if we choose $V$ as a mollifier approximating the Dirac delta function, formally the nonlocal term approximates 
$$
\nabla\cdot[\theta L^{d(m-2)} \nabla V* \rho]=\theta L^{d(m-2)} \nabla\cdot[\rho\nabla\rho] = \theta L^{d(m-2)} \Delta (\rho^2).
$$

 Therefore it is plausible that, at least when $\rho$ is bounded from above,  diffusion dominates when $m<2$ and the aggregation dominates when $m>2$.  Indeed we will show that, when $m<2$, the effect of the aggregation term is weak enough that it is possible to locally approximate solutions of Eq.\eqref{pde} with those of (PME).
 As a result, H\"{o}lder regularity of solutions of Eq.\eqref{pde} for $m<2$ follows. As for $m\geq 2$, we show that solutions are continuous ``uniformly 
% when relevant, explain later. 
in time'', based on the result of Dibenedetto (\cite{dib}).  For all $m>1$, we also show that the 
$L^{\infty}$ norm of
solution is uniformly bounded from above 
depending on the $L^1$ and $L^\infty$ norm of the initial data (see Theorem \ref{unif_bdd})  which is of independent interest.

\noindent $\bullet$
{\it Asymptotic behavior}
\hspace{.125 cm}
Our next result, partly an application of the first result, is on the asymptotic behavior of solutions of Eq.\eqref{pde} in the periodic domain $\mathbb{T}^d_L$.  We work in a periodic domain because, primarily, we are interested in finite volume problems and $\mathbb T_{L}^{d}$ provides the most convenient boundary conditions.
Even though asymptotic behavior for $m<2$ has been studied before in various references (e.g., \cite{s}, and \cite{hv} for a more singular interaction kernel) this is one of the first such result for these type of domains 
to the best of the authors' knowledge.  One difficulty specific to the periodic setting is that the radial symmetry is not preserved over time, and thus exact (non-constant)
solutions  -- always useful in these contexts -- are not readily available.  
We also point out that in the case $m\geq 2$, there exist solutions which assumes zero value, possibly with compact support.  Asymptotic behavior of such solutions are, in general, an interesting and difficult question, even for radial solutions in 
$\mathbb R^{d}$ (see \cite{ky}).

Intuitively, one expects that when the diffusion term is  ``dominant'' in Eq.\eqref{pde}, the solutions would converge to the constant solution as time goes to infinity. We show that this is indeed the case when $1<m<2$ and $\theta$ is sufficiently small.  However, with most interactions (specifically, $V$ being \textit{not} of positive type) there is a linear instability that sets in at some $\theta^{\sharp}= \theta^{\sharp}(m) < \infty$ which is 
determined by the minimal coefficient in the Fourier series of $V$ (see Section 4).
It is not hard to show that for all $m$, when 
$\theta > \theta^{\sharp}$, the functional in Eq.(\ref{energy})
has non--constant minimizers (and the constant solution is not a minimizer 
-- in fact, not even a {\it local} minimizer).  However as has been shown explicitly for $m = 1$ under reasonable conditions 
-- pertinently $d \geq 2$ -- this ``transition'' occurs at some 
$\theta_{\text{\textsc{t}}} < \theta^{\sharp}$ \cite{cp}.  Presumably, this argument holds in great generality.  It is therefore somewhat surprising that for $m = 2$ 
the transition occurs exactly at $\theta = \theta^{\sharp}$.

More precisely, for $m = 2$, we show that for $\theta<\theta^\sharp$ (the {\it subcritical} case), the constant solution is the only minimizer and is stable.  Indeed we can actually show that for all bounded initial data $\rho(x,0)$, the dynamical $\rho(\cdot, t)$ will converge to the constant solution $\rho_0$ exponentially fast in $L^2$-norm. 
See Section 4 for detailed discussion on critical and supercritical case. When $1<m<2$, the energy is no longer in the form of an $L^2$--norm, and our Fourier-transform based approach does not generate a transitional value for $\theta$.  However, when $\theta$ is sufficiently small, similar approach used by one of the authors in \cite{cp} yields that the constant solution is the global minimizer. Moreover, we show when $\theta$ is sufficiently small, the solution uniformly and exponentially converges to the constant solution.

Below we sketch an outline of our paper:   In Section 2 we first give a uniform upper bound for the weak solution to porous medium equation with a drift for $m>1$, then prove H\"older continuity of the weak solution when $1<m<2$.  In Section 3 we apply the H\"older continuity result to a nonlocal aggregation equation.  In Section 4 we use Fourier transform approach to study the nonlocal aggregation equation when $m=2$, and prove the exponential convergence of the weak solution in the subcritical case.  Analogous results for $1<m<2$ is established in Section 5.  When $1<m<2$, for $\theta$ sufficently small, we prove there is also exponential convergence. 

\section{H\"older Continuity of the Solution of PME with a Drift}

\noindent In this section, we study the regularity of the porous medium equation with a drift, where the drift  potential may depend on time:
\begin{equation}
\rho_t = \Delta(\rho^m) + \nabla \cdot (\rho \nabla \Phi) \quad \text{in $\Omega$}, \label{pmedrift0}
\end{equation}
with Neumann boundary condition on $\partial \Omega$.
Here we may assume $\Omega$ is a bounded open set in $\mathbb{R}^d$, where $d\geq 1$, but all the results in this section certainly hold for periodic domain $\mathbb{T}^d_L$ as well.  We assume $1<m<2$, the initial data $\rho(x,0)\in L^\infty(\Omega ) \cap L^1(\Omega )$, the potential $\Phi(x,t)\in C(\Omega \times \mathbb{R}^+)$, 
and that 
$\Phi(\cdot,t)\in C^2(\Omega )$ for all $t\geq 0$.  

Before even stating the main result, we will first prove that $\rho \in L^\infty(\Omega  \times \mathbb{R}^+)$.  When $\Phi$ does not depend on $t$, Bertch and Hilhorst in \cite{hilhorst} proved a uniform $L^\infty$ bound of $\rho$ by comparing $\rho$ with an explicit supersolution which does not depend on $t$.   When $\Phi$ is a function of both $x$ and $t$,  using arguments similar to those in \cite{kl}, we aquire an $L^\infty$ bound for $\rho$ which doesn't depend on $t$:

\begin{theorem} \label{unif_bdd}Suppose $m>1$.
Let $\rho$ be the unique weak solution of Eq.\eqref{pmedrift0} with Neumann boundary condition, with initial data $\rho(x,0)\in L^\infty(\Omega ) \cap L^1(\Omega )$. We assume that the potential $\Phi(x,t)$ satisfies $\Phi(x,t)\in C(\Omega \times \mathbb{R}^+)$, and $\Phi(\cdot,t)\in C^2(\Omega)$ for all $t$ with uniformly bounded norm.
Then there exists $M>0$, such that $\|\rho(\cdot,t)\|_{L^\infty(\Omega )}\leq M$ for all $t$, where $M$ depends on $\|\rho(x,0)\|_{L^\infty(\Omega)}$, $\|\rho(x,0)\|_{L^1(\Omega)},$ $ \sup_{t\in[0,\infty)}\|\Phi(\cdot, t)\|_{C^2(\Omega)}$, and $m$.
\end{theorem}

\begin{proof}
We begin with implementing the following scaling:  Let 
$$
\tilde{\rho}(x,t) = a^{\frac{1}{m-1}} \rho(x,at),
$$
where $0<a<1$.  Let us choose $a$ sufficiently small such that \begin{equation}
a< \min \Big\{ (\frac{1}{\|\rho(x,0)\|_{L^\infty(\Omega)}})^{m-1}, (\frac{c_0}{\|\rho(x,0)\|_{L^1(\Omega)}})^{m-1}, \frac{1}{\|\Phi\|_{C^2(\Omega)}} \Big\},
\end{equation}
where $c_0$ is a sufficiently small constant -- certainly less than 1 -- 
depending only on $m$ and $d$ and whose
precise value will be determined later.
By choosing $a$ in this  way, we have
both $\|\tilde{\rho}(x, 0)\|_{L^{1}(\Omega)} \leq c_{0} < 1$
and 
$\|\tilde{\rho}(x, 0)\|_{L^{\infty}(\Omega)} \leq 1$
and, moreover, that $\tilde{\rho}$ is a viscosity solution to the following PDE:
\begin{equation}
\tilde\rho_t = \Delta \tilde\rho^m + \nabla \cdot (\tilde\rho \nabla \tilde\Phi),  \label{after_scaling_pme_drift}
\end{equation}
where $\tilde\Phi := a\Phi$.  From the definition of  $a$ we know $\|\tilde\Phi(\cdot, t)\|_{C^2(\Omega)} \leq 1$ for all $ t$.

Our preliminary goal is to show $\|\tilde{\rho}(x, 1)\|_{L^{\infty}(\Omega)} \leq 1$;
then we can take $\tilde \rho(x,1)$ as the new initial data and iterate the argument to get a uniform bound for all time.

We will introduce another variable $v$, which is bigger than $\tilde \rho$ and is of order unity in $\Omega \times [0,1]$. Let $v$ be the viscosity solution to the following equation
\begin{equation}
\label{diffusion_v}
v_t = \nabla\cdot(mv^{m-1} \nabla v + v \nabla\tilde{\Phi}),  
\end{equation}
with initial data $v(x,0) = \tilde\rho(x,0) + \frac{1}{2}e^{-1}$.  Since $v$ solves the same equation as $\tilde \rho$ with bigger initial data, we can apply the comparison principle for the porous medium equation with drift, which was established in Theorem 2.21 of \cite{kl}. This comparison principle immediately implies $v(x,t) \geq \tilde{\rho}(x,t)$
for all $(x,t)$, hence it suffices to show $\|v(\cdot,1)\|_{L^{\infty}(\Omega)} \leq 1$.

One can check easily that $\tilde v(x,t) := [\|v(\cdot,0)\|_{L^{\infty}(\Omega)}] e^{K t}$ 
-- where 
$K := \sup_{t\in[0,\infty)}\|\tilde{\Phi}(\cdot, t)\|_{C^2(\Omega)}$ --
is a classical supersolution to Eq.\eqref{diffusion_v} and
hence also a viscosity supersolution.  
Noting that the initial data of $v$ satisfies, for all  $x$, $\frac{1}{2}e^{-1}\leq v(x,0) \leq 1+\frac{1}{2}e^{-1}$, the comparison principle gives the following upper bound for $v$:
$$
\|v(\cdot,t)\|_{L^{\infty}(\Omega)} \leq [\|v(\cdot,0)\|_{L^{\infty}(\Omega)}] e^{K t} \leq (1+\frac{1}{2} e^{-1}) e^{t}.
$$ 
Similarly we can find a classical subsolution which gives the lower bound
$$
\|v(\cdot,t)\|_{L^{\infty}(\Omega)} \geq [\|v(\cdot,0)\|_{L^{\infty}(\Omega)}] e^{-Kt} \geq \frac{1}{2} e^{-1} e^{-t}.
$$
Combining the two inequalities above, we have 
$$
v(x,t) \in [\frac{1}{2} e^{-2}, e+\frac{1}{2}] \quad \text{ for all } x \in \Omega, t\in [0,1].
$$

We would like to refine the estimate above and get a better estimate at $t=1$. By treating the diffusion coefficients $mv^{m-1}$ in Eq.\eqref{diffusion_v} as an \emph{a priori} function, -- which we denote by $b(x,t)$ --  then we may say that $v$ solves a linear equation of divergence form, where the diffusion coefficient is of (the order of) size unity:
\begin{equation}
v_t =\nabla \cdot (b(x,t) \nabla v+v\nabla \tilde\Phi),\label{pmedrift2}
\end{equation}
where $b(x,t) := mv^{m-1}(x,t) \in [ m(\frac{1}{2} e^{-2})^{m-1}, m(e+\frac{1}{2})^{m-1}]$ for all $x\in \Omega$, $t\in[0,1]$.

In particular, since Eq.\eqref{pmedrift2} is linear, we can decompose $v$ as $v_1+v_2$, such that $v_1$ solves Eq.\eqref{pmedrift2} with initial data $v_1(x,0)=\tilde\rho(x,0)$, and $v_2$ solves Eq.\eqref{pmedrift2} with initial data $v_2(x,0)= \frac{1}{2}e^{-1}$.  We claim that $v_1(x,1)$ and  $v_2(x,1)$ are both bounded by $\frac{1}{2}$, for all $x\in\Omega$.  

For $v_1$, first note that due to the divergence form of Eq.\eqref{pmedrift2}, the $L^1$ norm of $v_1$ is conserved, i.e. $\|v_1(\cdot,1)\|_{L^1(\Omega)}=c_0$. 
Since $b$ is bounded above and below away from zero, then
by \cite{lsu} (see Theorem 10.1, pp. 204), $v_1(\cdot,1)$ is H\"older continuous, where the H\"older exponent and coefficient do not depend on $c_0$, as long as $c_0<1$. So if we choose $c_0$ to be sufficiently small, we have $v_1(x,1)<\frac{1}{2}$ for all $x\in \Omega$.

For $v_2$, we can directly evaluate the necessary $L^\infty$ bounds: 
\begin{equation*}\sup_x v_2(x,1)\leq e^{\|\Delta \tilde \Phi\|_\infty} \sup_x v_2(x,0) \leq e~ \frac{1}{2}e^{-1}= \frac{1}{2}
\end{equation*}
(where again, on the basis of continuity, we may now talk about the supremum).

Combining the two estimates together, we have $\sup_x v(x,1)\leq 1$, which implies $\sup_x \tilde \rho(x,1)\leq 1$ from our discussion above.  Also, for $0<t<1$ we have $\rho(x,t) \leq v(x,t) \leq e+1/2$.  Then by treating $\tilde\rho(x,1)$ as initial data and iterating the same argument, we get $\sup_x \tilde\rho(x,t) \leq e+1/2$ for all $t$, i.e.,
$$\rho(x,t) \leq (e+\frac{1}{2})a^{-\frac{1}{m-1}} \quad \text{ for all } x\in \Omega, t\geq 0.$$
Now plugging in the definition of $a$ in the above and the bound becomes
$$\rho(x,t) \leq (e+\frac{1}{2}) \max \Big\{ \|\rho(x,0)\|_{L^\infty(\Omega)},~ \frac{\|\rho(x,0)\|_{L^1(\Omega)}}{c_0},~ \|\Phi\|^{\frac{1}{m-1}}_{C^2(\Omega)}\Big\} ~\text{ in } \Omega\times[0,\infty). $$
\end{proof}

%So:  I do not see why this is not uniform in volume.  In the language of the \textit{material density}  which integrates to the order of the volume, the $L^{\infty}$--norm is of order unity and the $L^{1}$ norm is large.  But this you handle by the small time step.  Well, not clear, because of crude use of the H\"older continuity; but, presumably with Lipschitz continuity, this would work.

\begin{remark}
In the statement of Theorem \ref{unif_bdd}, we assumed that $\Omega$ is a bounded open set, with Neumann boundary conditions.  The same proof also applies to Dirichlet boundary condition.  Indeed, the $L^\infty$ bound we obtained is independent with the size of $\Omega$, and the same proof works as well when $\Omega = \mathbb{R}^d$.
However, ostensibly, the $L^{\infty}$ norm of
$\rho$ should be of the order $L^{-d}$ and, even if true in the initial data, we cannot establish that this order is preserved at later times.
\end{remark}

Since $\rho(x,t)$ is uniformly bounded for all $(x,t)$, DiBenedetto has shown in \cite{dib} that $\rho(\cdot,t)$ is continuous uniformly in $t$:
\begin{theorem}[\cite{dib}] \label{unif_cont_rho}
For any $m > 1$, let $\rho$ be the weak solution to Eq.\eqref{pmedrift0} with initial data $\rho(x,0)\in L^\infty(\Omega ) \cap L^1(\Omega )$. Let the potential $\Phi(x,t)$ satisfy $\Phi(x,t)\in C(\Omega \times \mathbb{R})$, $\Phi(\cdot,t)\in C^2(\Omega)$ for all $t$, moreover $\sup_t \|\Phi(\cdot, t)\|_{C^2(\Omega)]} < \infty$. Then for all $\tau>0$, $\rho(x, t)$ is uniformly continuous in $\Omega \times [\tau, \infty)$, and the continuity is uniform in $x$ and $t$.
\end{theorem}

Now we want to show when $1<m<2$, for all $\tau >0$, $\rho(x,t)$ is uniformly H\"older continuous in space and time in $\Omega \times [\tau, \infty)$.  
Our main theorem of this section is stated as following:

\begin{theorem}\label{holder}
Let $1<m< 2$. Let $\rho$ be a viscosity solution of Eq.\eqref{pmedrift0}, with initial data $\rho(x,0)$.  We make the following assumptions on $\rho(\cdot, 0)$ and $\Phi$:
\begin{enumerate}
\item $\|\rho(\cdot, 0)\|_{\infty} \leq M_1$ and $\int_{\Omega} \rho(x,0) dx \leq M_1$.

\item $\Phi(x,t)\in C(\Omega \times \mathbb{R})$, and $\|\Phi(\cdot,t)\|_{C^2(\Omega )}\leq M_2$ for some $M_2>0$ for all $t\geq 0$.
\end{enumerate}  
Then for any $0<\tau<\infty$, $u$ is H\"older continuous in $\Omega \times[\tau,\infty)$, where the H\"older exponent and coefficient depends on $\tau, m,d, M_1$ and $M_2$.
\end{theorem}

\begin{proof}

To prove the H\"older continuity of $\rho$, our goal is to show that for any $(x_0, t_0) \in \Omega \times [\tau, \infty)$, 
\begin{equation}
\label{goal}
\text{{\large osc}}_{B(x_0, a^2)\times [t_0, t_0+a^4]} \rho \leq Ca^\gamma
\end{equation}
for some $C,\gamma >0$ not depending on $a$, for $a$ satisfying 
$0<a<\min\{\cfrac{2-m}{2c}, \sqrt{\tau}\}$ 
(where $c$ is a constant to be determined soon).

Bearing in mind that we want to zoom in on the profile and look at the oscillation in a small neighborhood, it makes sense to start with a parabolic scaling with scaling factor $a$.  Let
\begin{equation}
\tilde{\rho}(x,t) := \rho(ax, a^2t+(t_0-a^2))),
\end{equation}

%\noindent \textcolor{purple}{[3]  Are you sure?  Why is $t_{0}$ ``standing out''.}

\noindent and our goal Eq.\eqref{goal} would transform into
\begin{equation}
\label{osc_tilde_rho}
\text{{\large osc}}_{B(\frac{x_0}{a}, a) \times [1, 1+a^2]} \tilde \rho \leq Ca^\gamma.
\end{equation}
Here $\tilde \rho(x,t)$ is defined in the domain $\tilde \Omega \times [0,\infty)$, where $\tilde \Omega := \{x\in\mathbb{R}^d: ax\in\Omega\}$. 
and, it is noted, the early portion of the time domain had been omitted.  
We readily see that $\tilde \rho$ is the viscosity solution to
\begin{equation}\label{tilde_rho_pde}
\tilde \rho_t =\Delta \tilde \rho^m + \nabla \cdot ( \tilde \rho \nabla \tilde \Phi) \quad \text{in } \tilde \Omega \times [0,\infty).
\end{equation}
Here, the  initial data reads $\tilde \rho(x,0) = \rho(ax, t_0-a^2)$, which has an \emph{a priori} $L^\infty$ bound depending on $m, d, M_1, M_2$ due to Theorem \ref{unif_bdd}. 
Moreover, in the above
 $\tilde \Phi(x,t) := \Phi(ax, a^2t+(t_0-a^2)))$ and hence $|\nabla \tilde \Phi|$ is bounded by $aM_2$.   
%
%Formally speaking, the term $\nabla \cdot(\tilde\rho\nabla\tilde\Phi)$ in Eq.\eqref{tilde_rho_pde} is of order $a^2$, which motivates us 
%
%\noindent \textcolor{red}{[1]  Motivation not particularly clear.  It seems that all terms in Eq.(\ref{tilde_rho_pde}) are of order $a^{2}$.}
%
%\noindent \textcolor{red}{Two things come to mind:  (1) The sentence following 
%Eq.\eqref{tilde_rho_pde} which points out that the equation is the same ``but now'' 
%$\nabla \tilde \Phi$ is bounded by $aM_{2}$ -- which is small and therefore can be neglected.  (2)  A slightly more sophisticated scaling:  $\tilde{\rho}(y, \tau)
%:= a^{\ell}\rho(ay, a^{2}\tau)$  with $\ell > 0$.  Then, if $m > 1$, the potential term will indeed scale away.}
%
We wish to compare $\tilde \rho$ with $w$, where $w$ is the viscosity solution to the porous medium equation
\begin{equation}\label{w_pde}
w_t = \Delta w^m\quad \text{in } \tilde \Omega \times [0,\infty),
\end{equation}
with initial data $w(\cdot,0) \equiv \tilde \rho(\cdot,0)$.  Since Eq.\eqref{tilde_rho_pde} and Eq.\eqref{w_pde} only differ by the term $\nabla \cdot (\tilde \rho \nabla \tilde \Phi)$, we would expect 
\begin{equation}\label{difference_rho_w}
|\tilde \rho - w| \leq Ca^\beta  \text{ in } \tilde \Omega \times [1,2],
\end{equation} 
for some $C>0, 0<\beta<1$ depending on $m, d, M_1, M_2$.

The main part of this proof will be devoted to proving Eq.\eqref{difference_rho_w} is indeed true.  Without loss of generality, we can assume that $\tilde \rho(x,t)$ is a classical solution.  First, if the initial data $\tilde \rho(x,0)$ is uniformly positive, then $\tilde \rho(x,t)$ will be a classical solution for all time. This is because $\tilde \rho$ will stay positive for any time period $[0,T]$ (since $\inf_{\tilde{\Omega} \times [0,T]} \tilde \rho(x,t) \geq \exp(-t \sup_{t\in[0,T]}\|\Delta\tilde \Phi\|_\infty) \inf_{x\in\tilde{\Omega}} \tilde \rho(x,0)$), which implies that Eq.\eqref{tilde_rho_pde} is uniformly parabolic for $t\in[0,T]$ and hence the weak solution $\tilde \rho$ is classical.
%Since Eq.\eqref{difference_rho_w} is a comparison result, without loss of generality, we can assume the initial data $\tilde \rho(x,0)$ is positive and smooth, then $\tilde \rho$ would be a classical solution to Eq.\eqref{tilde_rho_pde}. For general data we can use approximation.}

For general initial data $\tilde \rho(x,0)$, we can use approximation as follows. Let $\tilde \rho_n$ and $w_n$ solve Eq.\eqref{tilde_rho_pde} and Eq.\eqref{difference_rho_w} respectively with initial data $\tilde \rho(x,0)+2^{-n}$; $n$ sufficiently large.   As discussed above, $\tilde \rho_n$ would be a sequence of classical solutions. If we can obtain $|\tilde\rho_n - w_n|<Ca^\beta$ for all $n$, (where $C, \beta$ doesn't depend on $n$), then Eq.\eqref{difference_rho_w} would hold for $\tilde \rho$ and $w$ as well, since as $n\to \infty$, comparison principle yields $\tilde\rho_n(x,t)\searrow \tilde \rho$ and $w_n(x,t) \searrow w$ uniformly in $x,t$.

%
%The main part of this proof will be devoted to proving Eq.\eqref{difference_rho_w} is indeed true.  Since Eq.\eqref{difference_rho_w} is a comparison result, without loss of generality, we can assume the initial data $\tilde \rho(x,0)$ is positive and smooth, then $\tilde \rho$ would be a classical solution to Eq.\eqref{tilde_rho_pde}. For general data we can use approximation.
%
%\noindent \textcolor{purple}{[$\star$]  I agree with Green.  We should elaborate.  Also, are there results to the effect that if it starts out classical it stays classical?}
%
%\textcolor{green}{This not particularly obvious: hope it clarifies later.}
%
%\textcolor{orange}{It is done by raising the initial data from $\tilde \rho(x,0)$ to $\tilde \rho(x,0)+2^{-n}$; then we would have $|\tilde\rho_n - w_n|<Ca^\beta$ for all $n$, where $C, \beta$ doesn't depend on $n$. Since $\tilde\rho_n(x,t)\searrow \tilde \rho$ and $w_n(x,t) \searrow w$ uniformly in $x,t$, we would obtain Eq.\eqref{difference_rho_w} by sending $n$ to infinity. Should I write that down?}

Note that one cannot directly compare $\rho$ with $w$, due to the fact that the term $\nabla \cdot (\tilde \rho \nabla \tilde \Phi)$ contains $\nabla \tilde \rho\cdot \nabla \tilde \Phi$ and hence does not have any \emph{a priori} bound.  In order to bound this term, it will help to change from the density variable $\tilde \rho$ to the pressure variable $\tilde u$. Let
\begin{equation*}
\tilde u=\frac{m}{m-1}\tilde \rho^{m-1},
\end{equation*}
 then Eq.\eqref{tilde_rho_pde} becomes
\begin{equation}
\tilde u_t = (m-1) \tilde u\Delta \tilde u + |\nabla \tilde u|^2 + \nabla \tilde u \cdot \nabla \tilde \Phi + (m-1) \tilde u \Delta \tilde \Phi,\label{rescale_pmedrift}
\end{equation}
which will enable us to use $|\nabla \tilde u|^2$ plus a constant to control the term $ \nabla \tilde u \cdot \nabla \tilde \Phi$:  
Recall that $|\nabla \tilde\Phi| < aM_2$, which gives us the following bound
\begin{equation*}
|\nabla \tilde{u}\cdot \nabla\tilde \Phi |
\leq aM_2 |\nabla \tilde u| 
\leq a[|\nabla \tilde{u}|^2 + \frac{1}{4}(M_2)^2].
\end{equation*}
Also, due to the fact that $(m-1)\tilde u(x,t) \leq C_1$ in $\tilde\Omega \times[0,2]$, (where $C_1$, which depends on  $m, d, M_1$ and $M_2$,  is related to the 
$L^{\infty}$ bounds on $\rho$) we obtain
\begin{equation*}
|(m-1)\tilde{u}\Delta \tilde\Phi |\leq a^2 C_1M_2 \leq aC_1 M_2.
\end{equation*}

Putting the above two bounds together, and by choosing $c$ such that $c>C_1 M_2+ (M_2/2)^2$, $\tilde u$ will satisfy the following inequality
\begin{equation}
\tilde u_t \geq (m-1)\tilde u \Delta \tilde u + (1-c a)|\nabla \tilde u|^2 -ca \quad \text{for all }x\in \tilde \Omega, t\in[0,2].\label{minus}
\end{equation}
Note that we assumed $a<(2-m)/(2c)$ in the beginning of the proof, we have $ca<(2-m)/2$.
%and similarly,
%\begin{equation}
%\tilde u_t \leq (m-1)\tilde u \Delta \tilde u + (1+c a)|\nabla \tilde u|^2 +ca\quad \text{for all }x\in \tilde \Omega,t\geq 0.\label{plus}
%\end{equation}

%We will deal with Eq.\eqref{minus} first. 
In order to make Eq.\eqref{minus} look similar to the porous medium equation in the pressure form, we apply the rescaling 
$u_1 = (1-ca)\tilde u$. Then $u_1$ satisfies
\begin{equation}
(u_1)_t \geq (m^- -1) u_1 \Delta u_1 + |\nabla u_1|^2 - ca(1-ca)\quad \text{for all }x\in \tilde \Omega,t\in [0,2], \label{wrong2}
\end{equation}  where 
\begin{equation}
m^- := \frac{m-1}{1-ca} + 1.  ~~(\text{hence $ca<(2-m)/2$ implies that } 1<m^- <2) \label{def_m*}
\end{equation}  Now Eq.\eqref{wrong2} has the same form as the porous medium equation in the pressure form, minus an extra constant term $ca(1-ca)$.  To take advantage of the existence and regularity results for equations with divergence form, we change the pressure variable back to the density variable (however here the power is $m^-$ instead of $m$), i.e., we define $\rho_1$ such that
\begin{equation}
(1-ca)\tilde u = u_1 = \frac{m^- }{m^-  -1} \rho_1^{m^-  \hspace{- 2 pt} -1}, \label{def_rho1}
\end{equation}
or in other words,
\begin{equation}\label{def2_rho1}
\rho_1 = \big(\frac{m}{m^-}\big)^{m^-} \tilde \rho^{1-ca} = \big(\frac{1+ca}{1+ca/m}\big)^{m^-} \tilde \rho^{1-ca}.\end{equation}
Due to the positivity of $\tilde u$, we know $\rho_1$ is positive as well.  Hence when we plug Eq.\eqref{def_rho1} into Eq.\eqref{wrong2}, after canceling a positive power of $\rho_1$ on both sides, we obtain
\begin{equation}
(\rho_1)_t > \Delta \rho_1^{m^- } - ca(1-ca) \rho_1^{2-m^- }\quad \text{in }\tilde \Omega \times [0,2] ,  \label{compare_2-m}
\end{equation}
Note that the term  $ca(1-ca)\rho_1^{2-m^-}$ has an \emph{a priori} upper bound: since $2-m^- >0$ and $\rho_1$ is given by Eq.
\eqref{def2_rho1}, we have $c(1-ca)\rho_1^{2-m^-} < M$, for some constant $M$ depending on $m, d, M_1, M_2$.

\noindent Let us denote by $\rho^-$ the weak solution of 
\begin{equation}
(\rho^-)_t = \Delta (\rho^- |\rho^-|^{m^- -1}) - Ma, \label{pme_minus}
\end{equation}
 with initial data the same as $\rho_1(x,0)$, which is
\begin{equation}\label{initial_rho-}
\rho^-(x,0) = \big(\frac{m}{m^-}\big)^{m^-} \tilde \rho(x,0)^{1-ca}
\end{equation} 

Since $\tilde\Omega$ is a bounded domain, we have $Ma \in L^p(\tilde \Omega)$ for all $p \geq 1$, and the existence of weak solution of Eq.\eqref{pme_minus} is guaranteed by Theorem 5.7 in \cite{v}.  That theorem also gives us a comparison result that, a.e., $\rho_1 \geq \rho^-$. 
 
 Moreover, note that the ``a.e.'' above can in fact be removed,  since both $\tilde \rho$ and $\rho^-$ are continuous in $\tilde \Omega \times [0,2]$: the continuity of $\tilde \rho$ is given by Theorem \ref{unif_cont_rho}, and the continuity of $\rho^-$ is given by Theorem 11.2 of \cite{dgv}.  
Therefore we have the following comparison between $\rho^-$ and $\tilde \rho$:
\begin{equation}
\rho^- \leq \big(\frac{m}{m^-}\big)^{m^-} \tilde \rho^{1-ca}~~\text{in }\tilde \Omega \times [0,2] \label{lowerbound}
\end{equation} 

Since $m/m^- = 1+O(a)$, and $\tilde \rho$ is bounded in $\tilde\Omega \times [0,2]$, Eq.\eqref{lowerbound} implies that  $\tilde \rho- \rho^- \geq -C_1 a$, where $C_1 $ depend on $m, d, M_1, M_2$.

Analogous to the definition to $\rho^-$, we define $\rho^+$ to be the weak solution of
\begin{equation}
(\rho^+)_t = \Delta ((\rho^+)^ {m^+}) + Ma, \label{pme_plus}
\end{equation}
with initial data
\begin{equation}\label{initial_rho+}
\rho^+(x,0) = \big(\frac{m}{m^+}\big)^{m^+} \tilde \rho(x,0)^{1+ca},
\end{equation} 
where
\begin{equation}
m^+ := \frac{m-1}{1+ca} + 1.  ~~(\text{hence } 1<m^+ <2) \label{def_m+}
\end{equation}
Then analogous argument would lead to $\tilde \rho - \rho^+ \leq C_1 a$. Summarizing, we have obtained
\begin{equation}\label{rho_rho-_rho+}
\rho^- - C_1 a \leq \tilde \rho \leq \rho^+ + C_1 a~~ \text{  in } \tilde \Omega \in [0,2],
\end{equation}
where $C_1 $ depends on $m, d, M_1, M_2$.
%
%We can deal with the other direction Eq.\eqref{plus} similarly.  Repeating the same argument would give us
%\begin{equation}
%\tilde u \leq u^+ :=\frac{m+ca}{(m-1)(1+ca)} (\rho^+)^{\frac{m-1}{1+ca}} ~~\text{a.e.} \label{upperbound}
%\end{equation} 
%where $\rho^+$ is the weak solution of
%\begin{equation}
%(\rho^+)_t = \Delta (\rho^+ |\rho^+|^{m^+ -1}) + aM, \label{pme_plus}
%\end{equation}

%\noindent with initial data
%\begin{equation}\label{initial_rho+}
%\rho^+(x,0) = [\frac{1+ca}{m+ca} (m-1) \tilde u(x,0)]^{\frac{1+ca}{m-1}}
%\end{equation} 
%and
%\begin{equation}
%m^+ := \frac{m-1}{1+ca} + 1,  ~~(\text{hence } 1<m^+ <2). \label{def_m+}
%\end{equation} 

To prove Eq.\eqref{difference_rho_w}, it suffices to show $|\rho^\pm -w| \leq O(a^\beta)$ for some $\beta>0$, which is proved in the following lemma.

\begin{lemma}\label{lemma} Let $1<m< 2$. 
Let $w$ be the viscosity solution of the porous medium equation
\begin{equation}
w_t = \Delta w^{m} \quad \text{in }\tilde \Omega \times [0,\infty) \label{pme_v}
\end{equation}where the initial data $w(x,0)$ satisfies $w(x,0) = \tilde \rho(x,0)$.

Let $\rho^-$ and $\rho^+$ be the weak solutions to Eq.\eqref{pme_minus} and Eq.\eqref{pme_plus} respectively, where  $0<a<(2-m)/(2c)$ is a small constant, and the initial data is given by Eq.\eqref{initial_rho-} and Eq.\eqref{initial_rho+}. Then
\begin{equation}
|\rho^\pm -w|\leq Ca^\beta  \quad\text{in }\tilde \Omega \times[1,2],
\end{equation}
where C and $\beta$ depends on $d, m, M_1, M_2$.
\end{lemma}

The proof of Lemma \ref{lemma} is the content of the appendix in Section \ref{proof}.  Putting Lemma \ref{lemma} and Eq.\eqref{rho_rho-_rho+} together, we obtain Eq.\eqref{difference_rho_w}, and we will use this to (immediately) prove Eq.\eqref{osc_tilde_rho}.

%By Lemma \ref{lemma}, together with the fact that $\rho^\pm$ and $v$ are both bounded in $t\in[1,2]$, we know
%\begin{equation*}
%\Big|\frac{m\pm ca}{(m-1)(1\pm ca)}(\rho^\pm)^{m^\pm -1} - \frac{m}{m-1} v^{m-1}\Big| \leq Ca^\gamma\quad\text{in }\tilde \Omega \times[1,2]
%\end{equation*}
% for some $C$ and $\gamma$, where $\gamma$ only depend on $\tau$,  and the $L^\infty$ bound.   
% Combine the above inequality with Eq.\eqref{lowerbound} and Eq.\eqref{upperbound}, we have
% \begin{equation}
%|\tilde u - \frac{m}{m-1} v^{m-1}| \leq Ca^\gamma,~~ \forall t\in[1,2].
%\end{equation}

Since $w$ solves the porous medium equation, Theorem 7.17 in \cite{v} gives us the H\"older continuity of $w$:
\begin{equation}\label{osc_pme}
\text{{\large osc}}_{B(x,a)\times [1,1+a^2]} w \leq C a^\alpha, \text{ for all } x\in\tilde\Omega,
\end{equation}
where $C$ and $\alpha$ depends on $\|w(\cdot, 0)\|_\infty$ (and hence depends on $m, d, M_1, M_2$).

By putting Eq.\eqref{osc_pme} and Eq.\eqref{difference_rho_w} together, we obtain
\begin{equation}\label{osc_pme}
\text{{\large osc}}_{B(x,a)\times [1,1+a^2]} \tilde \rho \leq C a^\gamma, \text{ for all } x\in\tilde\Omega,
\end{equation}
where $C$ depends on $m, d, M_1, M_2$, and $\gamma = \min\{\alpha, \beta\}$ (hence also depends on $m, d, M_1, M_2$).  Hence Eq.\eqref{osc_tilde_rho} is proved.
\end{proof}

%\paragraph{Remark} If $\rho_0$ is compactly supported, we know for any $T$, there exists $M_1$, such that $u(\cdot, t)$ is compactly supported in $[-M_1, M_1]$.  Then to show the H\"older continuity of $\rho$ in $(\tau,T)$, actually we only need $|\nabla \Phi|$ and $|\Delta \Phi|$ be bounded in $[-M_1,M_1]\times[0,T)$.

\begin{remark}
For $m \geq 2$, H\"older continuity of the solution to Eq.\eqref{pmedrift0} is still open.
Indeed, concerning the present approach -- which closely parallels that of \cite{kl}, \cite{k} -- when 
$m > 2$ we have that 
$m^{-} = 1 + (m-1)/(1 - ca) > 2$.  
Hence the ``inhomogeneous'' term in Eq.\eqref{compare_2-m},
which is proportional to 
$\rho^{(2-m^{-})}$, would actually be divergent in places where $\rho \to 0$.
This indicates that another approach will be required.
\end{remark}

\section{Application to Aggregation Equation with Degenerate Diffusion} \label{app_to_nonlocal}

In the following two sections, we study Eq.\eqref{pde}
in the domain $\mathbb{T}^d_L$, the $d$-dimension torus of scale $L$.  Here $\theta$ is a non-negative constant, and, of course, $*$ denotes convolution in $\mathbb{T}_L^d$.  We make the following assumptions on $V(x)$:
\begin{enumerate}
\item[\textbf{(V1)}] $V(x) = V(-x)$ for all $x\in\mathbb{T}_L^d$. 
\item[\textbf{(V2)}] $V(x) \in C^2( \mathbb{T}_L^d)$, with $\|V(x)\|_{C^2(\mathbb{T}_L^d)} = C$ for some constant $C < \infty$.
\end{enumerate}

\noindent Moreover, we have in mind $V: \mathbb R^{d} \to \mathbb R$ compactly supported 
with the diameter of the support smaller than $L$.  In particular we do not envision ``wrapping'' effects and
$\int_{\mathbb T_{L}^{d}}|V|dx$
may be regarded as independent of $L$.  

\vspace{.125 cm}

Our goal in this section is to show the H\"older continuity of the weak solution to Eq.\eqref{pde} for $1<m<2$, and uniform continuity of the weak solution when $m=2$.  First, we state the definition of weak solution to Eq.\eqref{pde} and a existence theorem from \cite{bs}.\\

\begin{definition}[Weak Solution] Let $m>1$, and let us assume that $\rho(x,0)$ is non-negative, with $\rho(x,0)\in L^\infty(\mathbb{T}_L^d)$ and consider a potential $V$ that satisfies the assumptions \textbf{(V1)} and \textbf{(V2)}. 
A function $\rho:\mathbb{T}_L^d\times[0,T]\to[0,\infty)$ is a weak solution to Eq.\eqref{pde} if 
$\rho\in L^\infty(\mathbb{T}_L^d\times[0,T])$,
 $\rho^m \in L^2(0,T,H^1(\mathbb{T}_L^d))$
 $($i.e., $\|\rho(\cdot, t)\|_{H^{1}(\mathbb T_{L}^{d}} \in L^{2}(0,T))$ and 
 $ \rho_t \in L^2(0,T,{H}^{-1}(\mathbb{T}_L^d))$ and for all test function $\phi\in H^1(\mathbb{T}_L^d)$, for almost all $t\in[0,T]$, 
\begin{equation}
<\rho_t(t),\phi> + \int_{\mathbb{T}_L^d} \nabla(\rho^m(t))\cdot \nabla\phi + \theta L^{d(2-m)} \rho(t)(\nabla V * \rho(t))\cdot \nabla\phi \hspace{1 pt}dx = 0.
\end{equation}
\end{definition}

In \cite{bs}, existence and uniqueness of weak solution are proved:
\begin{theorem}[Bertozzi-Slep\v{c}ev] \label{existence}
Let $m>1$ and consider $V$ that satisfies the assumptions \textbf{(V1)} and \textbf{(V2)}.  Let $\rho(x,0)$ be a nonnegative function in $L^\infty(\mathbb{T}_L^d)$. Then  the problem Eq.\eqref{pde} has a unique weak solution on $\mathbb{T}_L^d \times [0,T]$ for all $T>0$, and furthermore $\rho \in 
C(0, T, L^{p}(\mathbb{T}_L^d))$ for all $p\in [1,\infty)$.

\end{theorem}

By treating $\theta L^{d(2-m)} \rho*V$ as an \emph{a priori} potential, we can apply our results in Section 2, and obtain $L^\infty$ bound of $\rho$ which does not depend on $T$, together with uniform continuity of $\rho$, and H\"older continuity of $\rho$ for $1<m<2$.

\begin{theorem}
\label{uniform_bdd} Let $m>1$ and consider $V$ that satisfies the assumptions \textbf{(V1)} and \textbf{(V2)}.  Let $\rho(x,t)$ be the unique weak solution to Eq.\eqref{pde} given by Theorem \ref{existence}, with nonnegative initial data $\rho(x,0) \in C(\mathbb{T}_L^d)$, which satisfies $\int_{\mathbb{T}_L^d} \rho(x,0) dx = 1$.
%\noindent \textcolor{purple}{Again, I reemphasize: If a function is in  $C(\mathbb{T}_L^d$, it is automatically in $L^{\infty}$.}  \noindent \textcolor{green}{Or, perhaps, seeking results that are uniform in volume?}
Then $\|\rho(x,t)\|_{L^\infty(\mathbb{T}_L^d \times [0,\infty))}$ is bounded, where the bound only depend on $\sup_x \rho(x,0)$,  $\theta$, $\|V\|_{C^2}$ and $L$.
\end{theorem}
\begin{proof}
To begin with, note that Theorem \ref{existence} guarantees the existence and uniqueness of the weak solution to Eq.\eqref{pde}, which we denote by by $\rho$.  Now we treat $\Phi := \theta L^{d(2-m)}\rho*V$ as an \emph{a priori} potential, and we obtain the following estimate of $\Phi$ assumption \textbf{(V2)}:

\begin{eqnarray*}
\|\Phi(\cdot, t)\|_{C^2(\mathbb{T}_L^d)} &\leq& \theta L^{d(2-m)} \|\rho(\cdot, t)\|_{L^1(\mathbb{T}_L^d)} \|V\|_{C^2(\mathbb{T}_L^d)} \\
&\leq& \theta L^{d(2-m)}  \|V\|_{C^2(\mathbb{T}_L^d)} \\
&=& \theta L^{d(2-m)} C \quad \text{ for all } t\geq 0.
\end{eqnarray*}
We denote by $\rho_1$ the unique weak solution to the equation
\begin{equation}
(\rho_1)_t = \Delta \rho_1^m + \nabla \cdot (\rho_1\nabla\Phi)\label{pmedrift_temp}
\end{equation}
with initial data $\rho_1(\cdot, 0) \equiv \rho(\cdot, 0)$, where the existence and uniqueness is proved in \cite{hilhorst}. Theorem \ref{unif_bdd} implies $\sup_x \|\rho_1(\cdot, t)\|$ is bounded uniformly in $t$. 
Moreover, note that $\rho$ also satisfies the weak equation for Eq.\eqref{pmedrift_temp}, hence $\rho$ must coincide with $\rho_1$, which yields a uniform bound of $\rho$ which doesn't depend on time.
\end{proof}

Applying Theorem \ref{unif_cont_rho} to Eq.\eqref{pmedrift_temp}, we have the continuity of $\rho$ uniformly in $t$ for $m > 1$ -- in particular (in light of Theorem \ref{holder_m<2} below) for the case $m = 2$.

\begin{theorem} \label{unif_cont_rho_2}
Let $m>1$ and consider $V$ that satisfies the assumptions \textbf{(V1)} and \textbf{(V2)}.  Let $\rho(x,t)$ be the unique weak solution to Eq.\eqref{pde} given by Theorem \ref{existence}, with nonnegative initial data $\rho(\cdot, 0)$ satisfying 
$\|\rho(\cdot, 0)\|_{L^\infty(\mathbb{T}^d_L)}  < \infty$, 
and $\|\rho(\cdot, 0)\|_{L^1(\mathbb{T}^d_L)} =1$. Then for any $\tau>0$, $\rho$ is continuous in $\mathbb{T}_L^d \times [\tau,\infty)$, where the continuity is uniform in both $x$ and $t$.
\end{theorem}

\begin{proof}
Follows immediately from
the above reasoning, Theorem \ref{uniform_bdd} and Theorem \ref{unif_cont_rho}
 \end{proof}

%\begin{proof}
%Similar to the proof of Theorem \ref{uniform_bdd}, we again treat $\Phi := \theta L^{d(2-m)}\rho * V$ as an $a priori$ drift potential, and $\Phi$ satisfies $ \|\Phi(\cdot, t)\|_{C^2(\mathbb{T}_L^d)} \leq \theta L^{d(2-m)} C$  for all $ t\geq 0$.  Then $\rho$ coincides with the viscosity solution to Eq.\eqref{pmedrift_temp}.  Apply Theorem \ref{unif_cont_rho} to Eq.\eqref{pmedrift_temp}, then we have the desired result.\end{proof}

Applying  Theorem \ref{holder} to Eq.\eqref{pmedrift_temp}, 
with $\Phi = \theta L^{d(2-m)}\rho*V$
we have the H\"older continuity of $\rho$ for $1<m<2$.

\begin{theorem} \label{holder_m<2}
Let $1<m<2$ and consider $V$ that satisfies the assumptions \textbf{(V1)} and \textbf{(V2)}.  Let $\rho(x,t)$ be the unique weak solution to Eq.\eqref{pde} given by Theorem \ref{existence}, with nonnegative initial data  $\rho(x,0)$ satisfying 
$\|\rho(\cdot, 0)\|_{L^\infty(\mathbb{T}^d_L)} < \infty$, 
and $\|\rho(\cdot, 0)\|_{L^1(\mathbb{T}^d_L)} =1$. 
Then for any $\tau>0$, $u$ is H\"older continuous in $\mathbb{T}^d_L \times (\tau,\infty)$, where the H\"older exponent and coefficient depend on $\tau, m, d, \theta, L$ and $C$
and the $L^{\infty}$ norm of the initial condition.
\end{theorem}
\begin{proof}
Follows immediately from
the preceding  reasoning, Theorem \ref{uniform_bdd} and Theorem \ref{holder}

\end{proof}

\section{The Case $m = 2$:  Analysis Via Normal Modes}

In this section, we will use Fourier Transform to study the PDE in Eq.\eqref{pde}, and this method works best when $m=2$.  We continue to assume, without loss of generality
that $\|\rho(x,0)\|_{L^1(\mathbb{T}_L^d)}=1$, 
however from the perspective of 
\textit{functional analysis}, 
the homogeneity of the special case $m = 2$ makes even this 
stipulation redundant.

The  dynamics in Eq.\eqref{pde} is governed by gradient flow for the ``free energy'' functional 
\begin{equation}
\mathcal{F}_\theta(\rho) = \int_{ \mathbb{T}_L^d} \rho^2  + \frac{1}{2}\theta \rho(\rho*V) dx. \label{f}
\end{equation}
For the analysis of the functional $\mathcal{F}_\theta$, since we are assuming $\rho(x,0)$ integrates to 1, we shall denote by $\mathscr{P}$ the class of probability densities on 
$\mathbb{T}_L^d$ which also belong to $L^2(\mathbb{T}_L^d)$, i.e.
\begin{equation}\label{def_p}
\mathscr{P} := \{ f\in L^1(\mathbb{T}^d_L) \cap L^2(\mathbb{T}^d_L) : \|f\|_{L^1(\mathbb{T}^d_L)} = 1\}.
\end{equation}
Special to the case $m = 2$ is that the functional 
$\mathcal F_{\theta}(\cdot)$
can be expressed in a simpler form if we express $\rho$ in terms of its Fourier modes.
We write
$$
\hat{\rho}(k)  =  \int_{ \mathbb{T}_L^d} \rho(x)\text{e}^{-ik\cdot x}dx
$$
where $k$ is of the form
$k = \frac{2\pi}{L}\vec{n}$ with $\vec{n} \in \mathbb Z^{d}$.  With these conventions 
we have
$$
\rho(x)  =  \frac{1}{L^{d}}\sum_{k}\hat{\rho}(k)\text{e}^{ik\cdot x}
$$
and, in terms of these variables, Eq.\eqref{f} becomes
\begin{equation}
\label{f_fourier}
\mathcal{F}_\theta(\rho) = \frac{1}{L^{d}}\sum_{k} | \hat{\rho}(k)|^2 (1+\frac{1}{2}\theta \hat{V}(k)).
\end{equation}

On the basis of Eq.\eqref{f_fourier}, a salient value of $\theta$ emerges:  
We denote this value by $\theta^{\sharp}$, which is defined via
\begin{equation}
\label{theta_sharp_definition}
[\theta^{\sharp}]^{-1} := 
\frac{1}{2}\max_{k \neq 0}\{|\hat{V}(k)|; \hat{V}(k) < 0\}.
\end{equation}
Formally $\theta^{\sharp}$ may be designated as $+\infty$ in case $\hat V(k) \geq 0$ for all $k\neq 0$ -- i.e. if $V$ is (essentially) of positive type.  For the purposes of the present discussion, we shall assume otherwise.  
Different values of $\theta$ separate our problem into 3 cases:

\begin{enumerate}
\item (subcritical) When $\theta<\theta^{\sharp}$, we have 
$1+\frac{1}{2}\theta \hat{V}(k)>0$ for all 
$k\in\mathbb{Z}^d$, then under the restriction $\hat{\rho}(0) = 1$, it is manifest that global minimizer for $\mathcal{F}_\theta(\rho)$ in $\mathscr{P}$ is the constant solution 
\begin{equation}\label{def_rho_0}
\rho_{0}(x):=\frac{1}{L^d} \int_{\mathbb{T}^d_L} \rho(x,0) dx \equiv \frac{1}{L^d}.
\end{equation}

\item (critical) When $\theta=\theta^{\sharp}$,  we have still have $1+\frac{1}{2}\theta  \hat{V}(k)\geq 0$ for all $k\in\mathbb{Z}^d$ however now there is a set $\mathbb K^{\sharp}$ (containing at least two elements) defined by the condition that for
$k\in \mathbb K^{\sharp}$, $1 + \frac{1}{2}\theta^{\sharp} \hat{V}(k) = 0$.
In this case the global minimizers for $\mathcal{F}_\theta(\rho)$ in $\mathscr{P}$ take the form
\begin{equation}
\rho(x) = \rho_{0} + \sum_{k\in \mathbb K^{\sharp}} c_k 
\text{e}^{ik\cdot x},
\end{equation}
where $c_{-k}  =  \overline{c}_{k}$ and, of course, subject to the restriction that the resultant quantity is non--negative.

\item (supercritical) When $\theta > \theta^\sharp$, we have $1+ \frac{1}{2}\theta  \hat{V}(k)<0$ for some $k\in\mathbb{Z}^d$.  In this case the constant solution $\rho_{0}$ is not even a local minimizer of $\mathcal{F}_\theta$ in $\mathscr{P}$, let alone global minimizer.
\end{enumerate}

\begin{remark}
The above -- which is \textit{manifest} for $m = 2$ -- is in sharp contrast to the cases 
$m \neq 2$.  In particular, for general $m$ there is an analogous quantity 
$\theta^{\sharp}$ given by
$$
[\theta^{\sharp}]^{-1} := 
\frac{1}{m}\max_{k \neq 0}\{|\hat{V}(k)|; \hat{V}(k) < 0\}
$$
where items (1) -- (3) are suggested.  
However, the following was shown for $m = 1$ and, presumably holds for 
all $m \neq 2$:  While for 
$\theta < \theta^{\sharp}$, the constant solution
has ``some stability'' (c.f. \cite{cp} Theorem 2.11 for the case $m = 1$)
there is a $\theta_{\text{\tiny T}} < \theta^{\sharp}$
where global considerations come into play.  In particular, at 
$\theta = \theta_{\text{\tiny T}}$, there is a non--uniform minimizer for 
$\mathcal F_{\theta_{\text{\tiny T}}}(\cdot)$
which is degenerate with the uniform solution.  Moreover, for 
$\theta > \theta_{\text{\tiny T}}$ (which implies, in particular, at
$\theta = \theta^{\sharp}$)
the uniform solution is no longer a minimizer.
\end{remark}

\subsection{The subcritical case, when m=2}
In the subcritical case, the constant solution $\rho_0$ is the only global minimizer of $\mathcal{F}_\theta$ in $\mathscr{P}$.
Our goal in this section is to show for every non--negative initial data 
$\rho(x,0)\in L^\infty(\mathbb{T}_L^d)$ which integrates to 1, the weak solution $\rho(x,t)$ converges to $\rho_{0}$ exponentially in $L^2({T}_L^d)$ as $t\to\infty$, where $\rho_0$ is as given in Eq.\eqref{def_rho_0}.

By formally taking the time derivative of the free energy functional, a simple calculation indicates that e.g., at least for classical solutions to
Eq.(\ref{pde}),
 the free energy is always non--increasing:
\begin{equation}
 \label{dissipate}
\frac{d}{dt}\mathcal{F}_\theta(\rho) = -\int_{\mathbb{T}_L^d} \rho\big|\nabla(\frac{m}{m-1}\rho^{m-1}+\theta L^{d(2-m)}\rho*V)\big|^2 dx.
\end{equation}
In \cite{bs}, it is proved that Eq.\eqref{dissipate} is indeed true in the integral sense:
\begin{lemma}[Bertozzi--Slep\v{c}ev] 
\label{decrease}
Consider V that satisfies the assumptions \textbf{(V1)} and \textbf{(V2)}. Let $\rho(x,t)$ be a weak solution of Eq.\eqref{pde} in $\mathbb{T}_L^d \times [0,T]$. Then for almost all $\tau\in[0,T]$,
\begin{equation} \label{f_decay_int}
\mathcal{F}_\theta(\rho(\cdot,0)) -\mathcal{F}_\theta(\rho(\cdot,\tau)) \geq
 \int_0^\tau \int_{\mathbb{T}_L^d} \rho|\nabla(\frac{m}{m-1}\rho^{m-1}+\theta L^{d(2-m)} \rho*V)|^2 dx dt
\end{equation}
\end{lemma}
\vspace{-0.5cm}
\begin{remark} 
\label{GHT}
Theorem \ref{unif_cont_rho_2} implies that $\rho(\cdot, t)$ is a continuous function of $t$, hence $\mathcal{F}_\theta(\rho(\cdot, t))$ is continuous in $t$ as well.  Therefore Eq.\eqref{f_decay_int} indeed holds for all $\tau\in[0,T]$ and, moreover, 
Eq.(\ref{dissipate}) may be regarded as a \textit{differential inequality}.
\end{remark}

%We cite a weak semi-continuity property proved by Otto in \cite{otto} (see pages 165--166) which will be useful in the proof of Lemma \ref{decay}.  While the original lemma has space domain $\mathbb{R}^d$,  it works for $\mathbb{T}^d_L$ as well.
%\begin{lemma}[\cite{otto}]
%Given any non-negative 
%$u\in L^1(\mathbb{T}^d_L)$ and vector valued $A\in L^1(\mathbb{T}^d_L: \mathbb{R}^d)$, we have
%\begin{eqnarray*}
%\frac{1}{2} \int_{\mathbb{T}^d_L} \frac{|A|^2(x)}{ u(x)} dx =\sup_{\xi \in C^\infty(\mathbb{T}^d_L: \mathbb{R}^d)}\{ \int_{\mathbb{T}^d_L} A\cdot \xi dx - \frac{1}{2} \int_{\mathbb{T}^d_L} u|\xi|^2 dx \},
%\end{eqnarray*}
%where the fraction in the left hand side is understood in the sense
%\begin{equation*}
%\frac{w}{z} = \begin{cases} w/z & \emph{if} ~z\neq 0,\\
%0 & \emph{if } w=0 \emph{ and } z=0,\\
%\infty & \emph{if } w\neq 0 \emph{ and } z=0.
%\end{cases}
%\end{equation*}
%\end{lemma}
In the following lemma, we show when $\theta <\theta^{\sharp}$, the free energy will decay to the free energy of the global minimizer as $t\to\infty$.

\begin{lemma}
\label{YUY}
Suppose $m=2$ and 
consider $V$ that satisfies the assumptions \textbf{(V1)} and \textbf{(V2)}.
Further suppose that $\theta<\theta^{\sharp}$, where $\theta^\sharp$ is as given in Eq.\eqref{theta_sharp_definition} -- including $\theta^{\sharp} = \infty$ if $V$ is of positive type.
 Let $\rho(x,t)$ be the weak solution to Eq.\eqref{pde} on $[0,\infty)\times\mathbb{T}_L^d$, with non-negative initial data $\rho(x,0)\in L^\infty(\mathbb{T}_L^d)$ which integrates to 1.  Then $\mathcal{F}_\theta(\rho)\to \mathcal{F}_\theta(\rho_{0})$ as $t\to\infty$, where $\rho_0$ is the uniform solution $($as given in Eq.\eqref{def_rho_0}$)$.\label{decay}
\end{lemma}

\begin{proof} By Lemma \ref{decrease}, we know $\mathcal{F}_\theta(\rho(t))$ is a continuous and decreasing function of $t$, whose limit is bounded below by $\mathcal{F}_\theta(\rho_0)$, since $\rho_0$ is the global minimizer of $\mathcal{F}_\theta$ in $\mathscr{P}$ when $\theta < \theta^\sharp$.  Hence we can send $\tau$ to infinity in Eq.\eqref{f_decay_int}, which gives
\begin{equation}
\int_0^\infty \int_{\mathbb{T}_L^d} \rho|\nabla(2\rho+\theta\rho*V)|^2 dxdt <  \infty.
\end{equation}

Then there exists an increasing sequence of time $(t_n)_{n=1}^\infty$, where $\lim_{n\to\infty}t_n = \infty$, such that
 \begin{equation}
\label{PYM}
\lim_{n\to\infty} \int_{\mathbb{T}_L^d} \rho(x, t_n)|\nabla(2\rho(x, t_n)+\theta \rho(x, t_n)*V)|^2 dx =0.
\end{equation}
To avoid clutter, in what follows, we shall abbreviate
$\rho(\cdot, t_{n})$
by $\rho_n$.  
Recall that Theorem \ref{unif_bdd} gives us a uniform bound of $\|\rho_n\|_{L^\infty(\mathbb{R}^d)}$. In addition, by \cite{dib}, $(\rho_n)$ is uniformly equicontinuous, hence Arzel\`{a}-Ascoli Theorem enables us to find a subsequence of $\rho_n$ (which we again denote by $\rho_n$ for notational simplicity), and a continuous function $\rho_\infty$, such that

\begin{equation}\label{unif_conv_to_rho_infty}
\lim_{n\to\infty} \| \rho_n - \rho_\infty \|_{L^\infty(\mathbb{T}^d_L)} = 0,
\end{equation}
We next claim that
$\|\nabla \rho_n^{3/2}\|_{L^2(\mathbb{T}_L^d)}$ is bounded uniformly in $n$.  To prove the claim, we first note that
\begin{equation} 
\label{rho_n_estimate}
\int_{\mathbb{T}_L^d} \big| \frac{4}{3} \nabla \rho_n^{3/2} + \rho_n^{1/2} \nabla(\theta \rho_n*V)\big|^2 dx =  \int_{\mathbb{T}_L^d} \rho_n \big|2\nabla \rho_n + \nabla (\theta \rho_n*V)\big|^2 dx \to 0.
\end{equation}
To obtain the uniform $L^2$ bound for $\nabla \rho_n^{3/2}$, due to the triangle inequality, it suffices to prove a uniform $L^2$ bound for $\rho_n^{1/2} \nabla(\theta \rho_n*V)$, which is true since $\rho_n$ is uniformly bounded in $n$ and  $\|V\|_{C^2(\mathbb{T}_L^d)}< \infty$ due to \textbf{(V2)}, hence the claim is proved.

As a consequence of the claim, we obtain weak convergence of $\nabla \rho_n^{3/2} $ in $L^2$ (along another subsequence)  And, it is clear, the limit 
is just $\nabla \rho_\infty^{3/2}$ due to the uniform convergence of the $(\rho_{n})$.
(Moreover, this places $\nabla \rho_\infty^{3/2} \in L^2(\mathbb{T}^d_L: \mathbb{R}^d)$).
Thus:
\begin{equation}\label{L^2_weakly}
\nabla \rho_n^{3/2} \rightharpoonup \nabla \rho_\infty^{3/2} \text{ as } n\to\infty  \text{ weakly in } L^2(\mathbb{T}^d_L: \mathbb{R}^d).
\end{equation}

Let 
$$
B_n := \frac{4}{3} \nabla \rho_n^{3/2} + \rho_n^{1/2} \nabla(\theta \rho_n*V).
$$ 
Then  Eq.\eqref{unif_conv_to_rho_infty} and Eq.\eqref{L^2_weakly} 
and an additional uniform convergence argument identifying the weak limit of $\rho_n^{1/2} \nabla(\theta \rho_n*V)$,
\noindent implies that $B_n$ weakly converges to  $B_\infty$ in $L^2$, where $$B_\infty := \frac{4}{3} \nabla \rho_\infty^{3/2} + \rho_\infty^{1/2} \nabla(\theta \rho_\infty*V).$$ 
On the other hand, recall that Eq.\eqref{rho_n_estimate} gives us that $B_n \to 0$ strongly in $L^2$, thus we have $B_\infty$ is indeed 0 i.e.,
\begin{equation}\label{rho_infty_int}
 \int_{\mathbb{T}_L^d} \big| \frac{4}{3} \nabla \rho_\infty^{3/2} + \rho_\infty^{1/2} \nabla(\theta \rho_\infty*V)\big|^2 dx=
 \int_{\mathbb{T}_L^d} \rho_\infty \big|2\nabla \rho_\infty + \nabla (\theta \rho_\infty*V)\big|^2 dx = 0.
 \end{equation}

In particular, then,  
$\nabla(\rho_\infty+\frac{1}{2}\theta\rho_\infty*V)$
is zero a.e.~on the
support of $\rho_{\infty}$.  Now $\rho_{\infty}$ certainly admits a weak derivative
which, clearly, is non--zero only on the support of $\rho_{\infty}$.  Thus, from the 
preceding, we can write
\begin{equation}
\label{OUF}
 \int_{\mathbb{T}_L^d} \nabla\rho_\infty \cdot \nabla(\rho_\infty+\frac{1}{2}\theta  \rho_\infty*V) dx = 0.
 \end{equation}
 
 Now, we wish to express the above as a Fourier sum which requires some additional justification.  To this end we claim that $\rho_{\infty}$ is Lipschitz continuous 
 -- i.e., in $W^{1,\infty}(\mathbb T_{L}^{d})$ -- which places both entities 
 in $L^{2}(\mathbb T_{L}^{d})$ and vindicates the use of explicit formulas.  
 
 The equation $\nabla \rho_{\infty}  = -\frac{1}{2}\theta\nabla (V*\rho_{\infty})$ valid on the support of $\rho$ shows that in the various \textit{components} where $\rho_{\infty}$ is positive, 
 it is at least $C^2$.  Indeed, in general, Hypothesis 
\textbf{(V2)} immediately implies $\|\rho_\infty(x)*V\|_{C^2(\mathbb{T}^d_L)} \leq \|\rho_\infty\|_{L^1} \|V\|_{C^2(\mathbb{T}^d_L)}$ so whenever $\rho_{\infty}$ satisfies this
($m = 2$ version of the Kirkwood--Monroe) equation, we have Lipschitz continuity with uniform constant.  We shall denote this constant by $\kappa$. 
Now suppose that $x,y\in \mathbb T_{L}^{d}$ have $\rho_{\infty}(x)$ and
$\rho_{\infty}(y)$ positive.  Let us assume, ostensibly, that $x$ and $y$ belong to different components.  On the (shortest) line joining $x$ and $y$, let 
$z_{x}$ denote the first point, starting from $x$ that is encountered on the boundary of the component of $x$ and similarly for $z_{y}$.  Then
\begin{align}
\label{WQW}
|\rho_{\infty}(x) - \rho_{\infty}(y)|  & =  |\rho_{\infty}(x) - \rho_{\infty}(z_{x}) + \rho_{\infty}(z_{y}) - \rho_{\infty}(y)|
\notag
\\
& \leq |\rho_{\infty}(x) - \rho_{\infty}(z_{x})| + |\rho_{\infty}(z_{y}) - \rho_{\infty}(y)|
\notag
\\
& \leq \kappa[|x - z_{x}| + |y - z_{y}|] \leq \kappa|x - y|;
\end{align}
the first inequality due to $\rho_{\infty}(z_{x})  = \rho_{\infty}(z_{y}) = 0$
and the last inequality because all four points lie in order on the same line.  A similar argument can be used if, e.g., $\rho_{\infty}(x)$
is positive and $\rho_{\infty}(y)$ is zero.

All of this establishes enough regularity to unabashedly express Eq.(\ref{OUF}) in Fourier modes:

 \begin{equation}\label{rho_infty_fourier}
0 = \sum_{k}\frac{|k|^2}{L^d} |\hat{\rho}_\infty(k)|^2(1+\frac{1}{2}\theta \hat{V}(k)).
 \end{equation}
By the defining property of $\theta^{\sharp}$ we have $1+\frac{1}{2}\theta  \hat{V}(k)>0$ for all $k\neq 0$, thus Eq.\eqref{rho_infty_fourier} implies $\hat{\rho}_\infty(k)=0$ for all $k\neq 0$, i.e. $\rho_\infty\equiv\rho_{0}$.

Now, we may use the monotonicity in time of  $\mathcal{F}_\theta(\rho(t))$ and we finally have
$$
\lim_{t\to\infty} \mathcal{F}_\theta(\rho(t)) = \lim_{n\to\infty} \mathcal{F}_\theta(\rho_{n}) =  \mathcal{F}_\theta(\rho_\infty) = \mathcal{F}_\theta(\rho_{0})
$$
which is the stated claim.
\end{proof}

By combining the above result with the uniform continuity in time, we can show the solution will become uniformly positive after a sufficiently large time.
\begin{corollary}\label{uniform_pos}
Under the assumption of Lemma \ref{decay}, we have 
$$\lim_{t\to\infty} \|\rho(\cdot,t)-\rho_0\|_{L^\infty(\mathbb{T}_L^d)}= 0,$$ hence there exists $T>0$ depending on $\theta$, $\|V\|_{C^2(\mathbb{T}^d_L)}$ and $\rho(\cdot,0)$, such that $\rho(x,t)>\rho_0/2$ for all  $x\in\mathbb{T}_L^d, t>T$. 
\end{corollary}

\begin{proof}
%Lemma \ref{decay} yields that $\mathcal{F}_\theta(\rho(t)) - \mathcal{F}_\theta(\rho_0)\to 0$ as $t\to\infty$. Note that by Eq.\eqref{f_fourier}, the difference of free energy is comparable to the $L^2$ norm of the difference, i.e. 
%$$
%\mathcal{F}_\theta(\rho(t)) - \mathcal{F}_\theta(\rho_0)  \sim \|\rho(\cdot, t)-\rho_0\|^2_{L^2(\mathbb{T}_L^d)},
%$$
%which immediately implies that  $ \|\rho(\cdot,t)-\rho_0\|_{L^2(\mathbb{T}_L^d)} \to 0$ as  $t\to\infty$.  
%
%
%
%
%
%
%
%
%Note Theorem \ref{unif_cont_rho_2} implies continuity of $\rho$ uniformly in time, therefore convergence in $L^2(\mathbb{T}_L^d)$ implies $\|\rho(\cdot,t)-\rho_0\|_{L^\infty(\mathbb{T}_L^d)} \to 0$.
%\noindent \textcolor{purple}{================================} 
We prove the statement in the display.  Supposing that this is not the case.
Then there is a sequence of times, $(\tau_{n})$ and points $(y_{n})$ 
-- $y_{n} \in \mathbb T_{L}^{d}$ -- and a $\delta > 0$
such that
$$
|\rho(y_{n}, \tau_{n}) - \rho_{0}| > \delta.
$$
Now, going to a further subsequence, we have 
$y_{n} \to y_{\infty}$
(with $y_{\infty} \in \mathbb T_{L}^{d}$ by compactness).
But, along yet a further subsequence, not relabeled,
we have, according to the arguments of
Lemma \ref{YUY} that 
$\rho(\cdot, \tau_{n})$ is converging uniformly 
and the limit \textit{must} be $\rho_{0}$.  
Thus
$$
\lim_{n\to\infty}\rho_{n}(y_{n}, \tau_{n})
=\lim_{n\to\infty}[\rho_{n}(y_{n}, \tau_{n}) - \rho_{n}(y_{\infty}, \tau_{n})]
+\lim_{n\to\infty}\rho_{n}(y_{\infty}, \tau_{n})
= \rho_{0}
$$
in contradiction with the preceding display.
\end{proof}

\begin{theorem}\label{exp_conv}
Suppose $m=2$ and $\theta<\theta^{\sharp}$, where $\theta^\sharp$ is as given in Eq.\eqref{theta_sharp_definition}. Consider V that satisfies the assumptions \textbf{(V1)} and \textbf{(V2)}.  Let $\rho(x,t)$ be the weak solution to Eq.\eqref{pde} on $[0,\infty)\times\mathbb{T}_L^d$, with non--negative initial data $\rho(x,0)\in L^\infty(\mathbb{T}_L^d)$ which integrates to 1.  
Then $\mathcal{F}_\theta(\rho(t))$ decays exponentially to $\mathcal{F}_\theta(\rho_0)$, where the rate depend on $\rho(x,0)$. Moreover, $\|\rho(\cdot,t)-\rho_0\|_{L^2(\mathbb{T}_L^d)} \to 0$ exponentially, i.e.
\begin{equation*}
0\leq\mathcal{F}_\theta(\rho(t)) - \mathcal{F}_\theta(\rho_0) \leq C_1 \exp({-\frac{\rho_0 c'}{L^2} t}),
\end{equation*}
and
\begin{equation*}
\|\rho(t) - \rho_0\|_{L^2(\mathbb{T}^d_L)} \leq C_2 \exp({-\frac{\rho_0 c'}{L^2} t}),
\end{equation*}
where $c'$ and $C_1$ and $C_2$ depend on $\theta$, $V$ and $\rho(\cdot,0)$.
\end{theorem}

\begin{proof}
By Lemma \ref{uniform_pos}, there exist some $T>0$ depending on $\theta, V$ and $\rho(\cdot,0)$, such that $\rho(x,t)>\rho_0/2$ for all  $x\in\mathbb{T}_L^d, t>T$. Then for all $t_2> t_1>T$, Eq.\eqref{f_decay_int} becomes
\begin{eqnarray} 
\label{integral_inequality}
\nonumber\mathcal{F}_\theta(\rho(\cdot , t_1))-\mathcal{F}_\theta(\rho(\cdot ,t_2))&\geq& \int_{t_1}^{t_2} \int_{\mathbb{T}_L^d} \frac{\rho_0}{2} |\nabla(2\rho+\theta  \rho*V)|^2 dx dt\\
\nonumber&=& 2\rho_0 \int_{t_1}^{t_2}\frac{1}{L^d} \sum_{k} |k|^2 |\hat{\rho}(k)|^2 (1+\frac{1}{2}\theta \hat{V}(k)) ^2 dt\\
\nonumber&\geq& 
\rho_{0}c^{\prime}
 \int_{t_1}^{t_2} \frac{1}{L^d} \sum_{k \neq 0} | \hat{\rho}(k)|^2 (1+\frac{1}{2}\theta \hat{V}(k)) dt\\
&=& \rho_{0}c'\int_{t_1}^{t_2} (\mathcal{F}_\theta(\rho(\cdot, t))-\mathcal{F}_\theta(\rho_{0}) )dt, 
\end{eqnarray}
where $c' = 2 \min_{k\neq0}|k^{2}|(1+\frac{1}{2}\theta \hat{V}(k))$,
which is positive when 
$\theta<\theta^{\sharp}$.

In the spirit of Remark \ref{GHT} we may regard the above as a differential inequality for 
$g(t) :=\mathcal F_{\theta}(\rho(\cdot, t)) - \mathcal F_{\theta}(\rho_{0})$; the inequality reads
$$
-\frac{dg}{dt} \geq \rho_{0}c^{\prime}g(t).
$$
This immediately integrates to yield 
$g(t) \leq g(T)\text{exp}\{-\rho_{0}c^{\prime}(t- T)\}$ for $t \geq T$.  I.e., 
$$
\mathcal F(\rho(\cdot, t)) - \mathcal F(\rho_{0})  \leq  
C\text{e}^{-\rho_{0}c^{\prime}t}.
$$
Since $\mathcal{F}_\theta(\rho(\cdot,t)) - \mathcal{F}_\theta(\rho_0)$ is comparable to $\|\rho(t)-\rho_0\|_{L^2}$, we have $\|\rho(t)-\rho_0\|_{L^2}\to 0$ exponentially with the same rate.
\end{proof}

\begin{remark}
It is remarked that, via comparison to linearized theory, 
the above is essentially optimal.  (The results differ by a factor of two which comes from the definition of $T =: T_{1/2}$.  Using 
$T_{\epsilon} = 
\sup\{t > 0\mid||\rho(\cdot, t) - \rho_{0}||_{L^{\infty}(\mathbb T_{L}^{d})}
 > \epsilon \rho_{0}\}$, the \textit{long} time asymptotic rates are actually in complete agreement.)  Moreover, while for $L$ of order unity, the result stands:
 $c^{\prime}$ -- with or without an additional factor of two --
 might well be optimized at a wave number of order unity.  However, as 
 $L\to\infty$, it is clear that
 $$
 \min_{k\neq 0}|k|^{2}(1 + \frac{1}{2}\hat{V}(k))\to (\frac{2\pi}{L})^{2}
 (1 + \frac{1}{2}\hat{V}(0)).
 $$
 So, in particular, for large $L$ the rate scales as 
 $L^{-(d+2)}$ -- a result which may be an artifact of our normalization.
\end{remark}

\subsection{Some remarks on the supercritical case, when $m=2$}

When $\theta > \theta^\sharp$, we have $1+ \frac{1}{2}\theta \hat{V}(k_0)<0$ for some $k_0 = \frac{2\pi}{L}\vec n_0$, where $\vec n_0 \in \mathbb{Z}^d$. In other words, at least one of the coefficients of the free energy Eq.\eqref{f_fourier} is negative. In the next proposition we show that in this case the constant solution $\rho_0$ is not linearly stable.

\begin{proposition}Suppose $m=2$ and $\theta<\theta^{\sharp}$, where $\theta^\sharp$ is as given in Eq.\eqref{theta_sharp_definition}. Consider an interaction V that satisfies the assumptions \textbf{(V1)} and \textbf{(V2)}.   Then the constant solution $\rho_0$ is not a local minimizer of the free energy Eq.\eqref{f_fourier} in $\mathscr{P}$.
\end{proposition}

\begin{proof}
We choose $k_0= \frac{2\pi}{L}\vec n_0$ such that $1+ \frac{1}{2}\theta  \hat{V}(k_0)<0$,  where $\vec n_0 \in \mathbb{Z}^d$.  We add a small pertubation $\epsilon \eta$ to the constant solution $\rho_0$, where 
$$\eta := \cos(\frac{2\pi n_0 \cdot x}{L}).$$
Then
\begin{eqnarray*}
\mathcal{F}_\theta(\rho_0+\epsilon \eta) = \mathcal{F}_\theta(\rho_0) + L^d \epsilon^2(1+\frac{1}{2}\theta  \hat V(k_0)),
\end{eqnarray*}
which is strictly less than $\mathcal{F}_\theta(\rho_0)$ by the defining property of $k_0$.
\end{proof}
\begin{remark}\label{touch_zero} In fact, using the same perturbation term in the proof, we would know that when $\theta>\theta^\sharp$, any strictly positive function is not a local minimizer of the free energy Eq.\eqref{f_fourier}.
\end{remark}

In the supercritical case, while Eq.\eqref{f_fourier} immediately implies that $\rho_0$ is not a local minimizer of $\mathcal{F}_\theta$ in $\mathscr{P}$, it gives us little information about what is the global minimizer. The difficulty comes from the restriction $\rho(x)\geq 0$ for all $x$, which evidently plays an important role in the supercritical case, since any minimizer should touch zero somewhere due to Remark \ref{touch_zero}.   After Fourier transform, the non-negativity of $\rho$ actually gives us infinite numbers of restrictions, which causes the difficulty.

\section{Exponential decay for $1<m<2$ and 
%$\theta L^{d(2-m)} \ll 1$ 
weak interaction}
In this section, we continue our study of Eq.(\ref{pde}) with $m\in (1,2)$ and
here we will assume that  $\theta$ is ``small''.  Unfortunately, $\theta$ will 
\textit{not} be uniformly small in volume.  In particular, we shall require 
$\theta L^{d(2-m)}$ to be a small number of order unity and, under these conditions
we shall acquire all the results of the previous section.  We claim that without additional (physics based) assumptions -- in particular H--stability of the interaction -- the above condition is essentially optimal.  Specifically, our cornerstone result of a unique stationary state does not hold for non--H--stable interactions when 
$\theta L^{d(2-m)}$ is a sufficiently \textit{large} number of order unity.   
However, from an \ae sthetic perspective, this uniqueness result 
is the sole instance where $\theta L^{d(2-m)}$ must be considered small. 
In the aftermath of Proposition \ref{YGY} 
and its corollary, we will only require 
$\theta$ itself to be a small quantity.

We start with a priliminary result (which is, actually, just a quantitative version of 
the argument used in Lemma \ref{YUY} in the vicinity of Eq.(\ref{WQW})).
\begin{proposition}
\label{YGY}
Consider an interaction V that satisfies the assumptions \textbf{(V1)} and \textbf{(V2)}.
Let 
$$
\varepsilon_{0} := \theta L^{d(2-m)}
$$
be a sufficiently small number of order unity.  Let $\rho$ denote any solution to the 
Kirkwood--Monroe equations which here read, whenever $\rho > 0$, 
$$
\nabla \rho^{m-1}  =  -\varepsilon_{0}\frac{m-1}{m}\rho*\nabla V
$$
and let 
$$
R := \|\rho\|_{L^{\infty}(\mathbb T_{L}^{d})}.
$$
Then if $\varepsilon_{0}$ is a small number of order unity then $R$ is also a small number of order unity (if $L$ is large).  In particular, 
$$
R \leq 
\kappa_{4}\max\{
[\varepsilon_{0}]^{\frac{d}{d(m-1)+1}},
L^{-d}
\}
$$
with $\kappa_{4}$ a constant of order unity.
\end{proposition}
\begin{proof}
From the mean--field equations,
$$
|\nabla\rho^{m-1}| \leq  \frac{m-1}{m}
\varepsilon_{0}\int_{\mathbb T_{L}^{d}}
|\nabla V(x-y)|\rho(y)dy
\leq \frac{m-1}{m}\|V\|_{C_{1}}\varepsilon_{0} =: \kappa_{1}\varepsilon_{0}.
$$

Let $x_{0}$ mark the spot where $\rho$ achieves $R$.  Then, for all $x$,
$$
\rho^{m-1}(x)  \geq 
R^{m-1} - \kappa_{1}\varepsilon_{0}|x-x_{0}|.
$$
Thus, if $r$ is the length scale of the region about $x_{0}$ where
$\rho^{m-1}$ exceeds, a.e., $\frac{1}{2}R^{m-1}$ 
we have 
$$
r \geq
\frac{R^{m-1}}{2\kappa_{1}\varepsilon_{0}}
$$
provided the right hand side does not exceed $L$.  Otherwise, obviously, $r = L$.
Since $\rho$ integrates to unity we have, assuming $r < L$,
$$
1 = \int_{\mathbb T_{L}^{d}}\rho dx
\geq \kappa_{2}r^{d}R
\geq \kappa_{2}
\frac{R^{d(m-1) + 1}}{(2\kappa_{1}\varepsilon_{0})^{d}} =: \frac{1}{\kappa_{3}^{d}}
\frac{1}{\varepsilon_{0}^{d}} R^{d(m-1) +1}
$$
(with $\kappa_{2}$ a geometric constant of order unity) and otherwise we acquire the mundane bound.  After a small step, the stated bound is obtained with an appropriate definition of $\kappa_{4}$.  
\end{proof}

With the above in hand, we can establish that $\rho_{0}$ is the unique stationary solution.  We start with 
\begin{corollary}
\label{uniqueness_KM}
Under the conditions stated in Proposition \ref{YGY}, 
if $\varepsilon_{0}$ is sufficiently small -- but of order unity independent of $L$ --
the unique solution to the 
mean--field equations is $\rho = \rho_{0}$.
\end{corollary}
\begin{proof}
From the mean--field equation, we may write
$$
0 = 
\int_{\mathbb T_{L}^{d}}\nabla\rho\cdot\nabla(\rho^{m-1} + \varepsilon_{0}\frac{m-1}{m}
\rho*V
)dx.
$$
By recapitulating the Lipchitz continuity that was featured in the vicinity of Eq.(\ref{WQW})
we have full justification to manipulate classically under the integral.  Letting 
$R_{\varepsilon_{0}}$ denote the upper bound on the $L^{\infty}$ norm of $\rho$
that was featured in Proposition \ref{YGY}.  Then, pointwise a.e.~on the support of $\rho$,
$$
\nabla \rho\cdot \nabla \rho^{m-1}  =  \frac{m-1}{\rho^{2-m}}|\nabla \rho|^{2}
\geq \frac{m-1}{R_{\varepsilon_{0}}^{2-m}}|\nabla \rho|^{2}
$$
since, we remind the reader, $2-m > 0$.
In other words,
$$
0 \geq \int_{\mathbb T_{L}^{d}}\frac{1}{R_{\varepsilon_{0}}^{2-m}}|\nabla \rho|^{2}
+ \frac{\varepsilon_{0}}{m}\nabla \rho\cdot \nabla (\rho*V)dx.
$$
We can again go to Fourier modes and the above reads
$$
0 \geq
\sum_{k\neq 0}k^{2}|\hat{\rho}(k)|^{2}[\frac{1}{R_{\varepsilon_{0}}^{2-m}} + 
\frac{\varepsilon_{0}}{m}\hat{V}(k)].
$$
For $\varepsilon_{0}$ sufficiently small (but of order unity independent of $L$)
the coefficient of $|\hat{\rho}(k)|^{2}$ is positive for all terms so the later must vanish identically.  The desired result is proved.
\end{proof}

Based on the fact that $\rho_0$ is the unique stationary solution, in the next lemma we prove that $\rho(\cdot, t)$ will converge to $\rho_0$ uniformly, 
(but not with a quantitative estimate on the \textit{rate}.)

\begin{lemma}\label{unif_conv_m<2}
Suppose the 
conclusions in Corollary \ref{uniqueness_KM}
are satisfied.  Let $\rho(x,t)$ be the weak solution to Eq.\eqref{pde} on $[0,\infty)\times\mathbb{T}_L^d$, with non-negative initial data $\rho(x,0)\in L^\infty(\mathbb{T}_L^d)$ which integrates to 1.  Then $\sup_x |\rho(\cdot, t)-\rho_0|\to 0$ as $t \to \infty$.
\end{lemma}

\begin{proof}
This is more or less identical to the proof of 
Corollary \ref{uniform_pos} based on Lemma \ref{YUY} so we shall be succinct.  
Assuming the result false, we could find a sequence of times 
$t_{n}\to\infty$
and points
$x_{n}\to x_{\infty} \in \mathbb T_{L}^{d}$
such that 
$\rho(\cdot,t_{n})$ converges uniformly
and yet 
$|\rho(x_{n}, \tau_{n}) - \rho_{0}| > \delta$.
So, denoting by 
$\rho_{\infty}(\cdot)$
the uniform limit, we would have
$|\rho_{\infty}(x_{\infty}) - \rho_{0}| > \delta$.

Hence, since $\rho_{\infty}$ is continuous, it is definitively
\textit{not} equal to $\rho_{0}$.  
However, any subsequential limit must satisfy the mean--field equation
and by Corollary \ref{uniqueness_KM} this is uniquely $\rho_{0}$
 in contradiction with the preceding.  
 This completes the proof.
\end{proof}

In the next lemma, we show that once $\rho$ and $\rho_0$ becomes comparable, $\mathcal F_{\theta}(\rho) - \mathcal F_{\theta}(\rho_{0})$ also becomes comparable with $L^{d(2-m)}\|\rho-\rho_0\|_{L^2(\mathbb{T}_L^d)}$. %Finally, we show that under the relevant condition $t > T$ -- when 
%$\|\rho - \rho_{0}\|_{\mathbb T_{L}^{d}} < \frac{1}{2}\rho_{0}$ --
%that the difference 
%$\mathcal F_{\theta}(\rho) - \mathcal F_{\theta}(\rho_{0})$ bounds from above the $L^{2}$ difference of $\rho - \rho_{0}$.  
 Indeed, as alluded to earlier, this will be proved under the weaker assumption 
 that $\theta$ --
not $\theta L^{d(m-2)}$ -- is small.  We start with:

\begin{lemma}\label{lemma_comparable}
Suppose that $\theta>0$ is sufficiently small (but of order unity independent of $L$). Let $\rho$ be such that
$\|\rho - \rho_{0}\|_{\mathbb T_{L}^{d}} < \frac{1}{2}\rho_{0}$. Then we have
\begin{equation}\label{comparable_energy}
\alpha L^{d(2-m)}\|\rho - \rho_{0}\|^{2}_{L^{2}(\mathbb T_{L}^{d})} \leq \mathcal F_{\theta}(\rho) - \mathcal F_{\theta}(\rho_{0}) \leq
\beta L^{d(2-m)}\|\rho - \rho_{0}\|^{2}_{L^{2}(\mathbb T_{L}^{d})}
\end{equation}
for some $\alpha, \beta > 0$ of order unity.  
\end{lemma}
\begin{proof}
First, by any number of methods we have
$$
\int_{\mathbb T_{L}^{d} \times \mathbb T_{L}^{d}}
\rho(x)\rho(y) V(x-y)dxdy \geq
-K_{V}\|\rho - \rho_{0}\|^{2}_{L^{2}(\mathbb T_{L}^{d})};
$$
e.g., we may take, using the Fourier decomposition, $K_{V}  =  [\theta^{\sharp}]^{-1}$.  Similarly for a corresponding \textit{upper} bound with a positive constant.
Let us turn to the entropic--like terms.

Writing $\rho = \rho_{0}(1 + \eta)$, our assumption implies that $|\eta| \leq \frac{1}{2}$. From this it is easy to verify that, pointwise, 
$$
(1 + \eta)^{m}  \geq  
1 + m\eta + \frac{m(m-1)}{2}(\frac{2}{3})^{2-m}\eta^{2} := 1 + m\eta + a\eta^{2},
$$
and for the other direction we have
$$
(1 + \eta)^{m}  \leq  
1 + m\eta + \frac{m(m-1)}{2}(3)^{2-m}\eta^{2} := 1 + m\eta + b\eta^{2}.
$$
Thence
$\rho^{m} - \rho_{0}^{m}  =  \rho_{0}^{m}[(1 + \eta)^{m} - 1]
= \rho_{0}^{m}[(1 + \eta)^{m} - 1 - m\eta + m\eta]
\geq \rho_{0}^{m}[m\eta + a \eta^{2}]$.
So
$$
\int_{\mathbb T_{L}^{d}}(\rho^{m} - \rho_{0}^{m})dx
\geq a\rho_{0}^{m}\|\eta\|^{2}_{L^{2}(\mathbb T_{L}^{d})}
= a L^{d(2-m)}\|\rho - \rho_{0}\|^{2}_{L^{2}(\mathbb T_{L}^{d})},
$$
and similarly we have
$$
\int_{\mathbb T_{L}^{d}}(\rho^{m} - \rho_{0}^{m})dx
\leq b L^{d(2-m)}\|\rho - \rho_{0}\|^{2}_{L^{2}(\mathbb T_{L}^{d})}.
$$
Combining this with the bounds on the energy term, the stated claim 
has been established.
\end{proof}

Finally, in the next theorem, we prove that the free energy decays exponentially
to its minimum value.

\begin{theorem}\label{exp_conv_m<2}
Suppose the conclusions acquired in 
Corollary \ref{uniqueness_KM}
are satisfied
and suppose that $\theta$ is a sufficiently small number which is of the order of unity.
 Let $\rho(x,t)$ be the weak solution to Eq.\eqref{pde} on $[0,\infty)\times\mathbb{T}_L^d$, with non-negative initial data $\rho(x,0)\in L^\infty(\mathbb{T}_L^d)$ which integrates to 1.  Then $\mathcal{F}_\theta(\rho(t))$ decays exponentially to $\mathcal{F}_\theta(\rho_0)$. More precisely, 
\begin{equation}\label{free_energy_decay_m<2}
\mathcal F(\rho(\cdot, t)) - \mathcal F(\rho_{0})  \leq  C_1
\text{e}^{-\rho_{0}^{m-1}c^{\prime}t}
\end{equation}
for various constants $c^{\prime}$ and $C_1$.  Similarly for the $L^{2}$--norm of
$(\rho - \rho_{0})$ with a different prefactor.  
\end{theorem}
\begin{proof}
According to Lemma \ref{unif_conv_m<2}, there exist some $T>0$ depending on $\theta, L, V$ and $\rho(\cdot,0)$, such that 
$|\rho(x,t) -\rho_0| < \frac{1}{2}\rho_{0}$ for all  $x\in\mathbb{T}_L^d, t>T$. Then for all $t_2> t_1>T$, 
we manipulate the integrand on the right hand side of 
Eq.(\ref{f_decay_int}) -- the lower bound on
$\mathcal{F}_\theta(\rho(\cdot , t_1))-\mathcal{F}_\theta(\rho(\cdot ,t_2))$:
\begin{align}
%\label{}
\int_{\mathbb T_{L}^{d}}\rho
|\nabla\frac{m}{m-1}\rho^{m-1} + \theta & L^{d(2-m)}\nabla  \rho* V|^{2}dx
\geq 
\notag
\\
&\int_{\mathbb T_{L}^{d}}\rho
\left [
\frac{1}{2}|\nabla\frac{m}{m-1}\rho^{m-1}|^{2}
- |\theta L^{d(2-m)}\nabla \rho* V|^{2}
\right ]dx  
\notag
\\
= & \int_{\mathbb T_{L}^{d}}
\left [
\frac{1}{2}m^{2}\rho^{2m-3}|\nabla \rho|^{2}
- \rho \theta^{2}L^{2d(2-m)}|\nabla (\rho*V)|^{2}
\right]dx
\notag
\\
\geq & \int_{\mathbb T_{L}^{d}}
\left[
g\rho_{0}^{2m-3}|\nabla \rho|^{2} - 
\frac{3}{2}\rho_{0}\theta^{2}L^{2d(2-m)}|\nabla (\rho*V)|^{2}
\right]dx
\end{align}
where the value of $g$ -- which is always of order unity -- 
depends on whether $2m-3$ is positive or not.
Note that all terms are proportional to 
$\rho_{0}^{2m - 3} = 
\rho_{0}^{m-1}L^{d(2-m)}$.

Going to Fourier modes, the final (spatial) integral in the above string becomes
$$
\rho_{0}^{m-1}L^{d(2-m)}\cdot
\frac{1}{L^{d}}\sum_{k}
k^{2}|\hat{\rho}(k)|^{2}[g - \frac{3}{2}\theta^{2}|\hat{V}(k)|^{2}]
$$
where, for sufficiently small $\theta$, we may assert that the summand is positive.  

We thus have
\begin{align}
\label{ZFD}
\mathcal{F}_\theta(\rho(\cdot , t_1))
& -
\mathcal{F}_\theta(\rho(\cdot ,t_2))
\geq
\notag
\\
&\rho_{0}^{m-1}
\beta
c^{\prime}
\int_{t_{1}}^{t_{2}}\int_{\mathbb T_{L}^{d}}L^{d(2-m)}(\rho - \rho_{0})^{2}dxdt
\hspace{2 pt}
\geq
\hspace{2 pt}
\rho_{0}^{m-1}c^{\prime}
\int_{t_{1}}^{t_{2}}
[\mathcal{F}_\theta(\rho(\cdot , t)) - \mathcal{F}_\theta(\rho_{0})]
dt
\end{align}
where in the above, $\beta$ is the constant from
Lemma \ref{lemma_comparable}
which has been conveniently absorbed into the definition of $c^{\prime}$:
$$
c^{\prime}\beta :=  \min_{k\neq 0}[k^{2}(g - \frac{3}{2}
\theta^{2}|\hat{V}(k)|^{2})]
$$
and in the final step we have used 
Lemma \ref{lemma_comparable}.

Note that Eq.\eqref{ZFD} has the same form as Eq.\eqref{integral_inequality}  therefore we can again treat it as a differential inequality as in the proof of Theorem \ref{exp_conv}.  We obtain that
$$
\mathcal F(\rho(\cdot, t)) - \mathcal F(\rho_{0})  \leq  C_1
\text{e}^{-\rho_{0}^{m-1}c^{\prime}t}.
$$
A further application of 
Lemma \ref{lemma_comparable} implies a similar result for the $L^{2}$--norm of
$(\rho - \rho_{0})$
and the proof is finished.
\end{proof}

\begin{remark}
Here as in the case $m = 2$, 
when $L$ is large, 
$c^{\prime} \propto L^{-2}$ and we 
have the large $L$ scaling of the rate proportional to
$L^{-(2 + d(m-1))}$ in agreement with a perturbative analysis.  However in this case, our arguments
do not provide agreement with the constant of proportionality.  
We also note that by Theorem \ref{holder} we have that $\rho(\cdot, t)$ is uniformly H\"older continuous in space and time for all $t\geq T$, where the H\"older coefficient and exponent depends on $\theta$, $L$ and $V$.  
Thus we can bound, the $L^{\infty}$--norm
of $\rho - \rho_{0}$ by some power of its $L^{2}$--norm.  
Hence the exponential convergence of $\|\rho-\rho_0\|_{L^2(\mathbb{T}_L^d)}$ implies the exponential convergence of $\|\rho-\rho_0\|_{L^\infty(\mathbb{T}_L^d)}$.
However 
a bound along these lines is ``even more'' non--optimal since the two norms should, presumably, differ by a factor of
$L^{d}$.
\end{remark}

%\begin{corollary}
%Under the same conditions stated in Theorem \ref{exp_conv_m<2}, we have that $\|\rho(\cdot, t)-\rho_0\|_{L^\infty(\mathbb{T}_L^d)}$ decays exponentially to 0:
%\begin{equation*}
%\|\rho(\cdot, t)-\rho_0\|_{L^\infty(\mathbb{T}_L^d)} \leq C_2 \exp({-\frac{\rho_0^{m-1} c_2}{L^2} t}),
%\end{equation*}
%where $c'$  depends on $\theta, L$ and $V$, and $C_2$ depends on $\theta, L, V$ and $\rho(\cdot,0)$.
%\end{corollary}
%
%\begin{proof}
%From the proof of Theorem \ref{exp_conv_m<2}, we know that there exists some $T>0$, such that $\rho_0/2 < \rho(x, t)<2\rho_0$ for all $x\in \mathbb{T}_L^d$ and $t>T$.  When $\rho$ and $\rho_0$ become comparable, Lemma \ref{lemma_comparable} says that $\mathcal{F}_\theta(\rho) - \mathcal{F}_\theta(\rho_0)$ and $L^{d(2-m)} \|\rho-\rho_0\|_{L^2(\mathbb{T}_L^d)}$ are also comparable, hence Eq.\eqref{free_energy_decay_m<2} implies that 
%$$\|\rho(\cdot, t) - \rho_0\|_{L^2(\mathbb{T}_L^d)} \leq C' \exp({-\frac{\rho_0^{m-1} c'}{L^2} t}),$$
%where $c'$ is the same as in Theorem  \ref{exp_conv_m<2}, and $C'$ depends on $\theta, L, V$ and $\rho(\cdot,0)$. 
%
%
%\end{proof}

\section{Appendix:  Proof of Lemma \ref{lemma}}\label{proof}

\begin{proof}[Proof of Lemma \ref{lemma}]
We will do the comparison between $\rho^-$ and $w$ first; the comparison between $\rho^+$ and $w$ can be done in the same way.  

%Since $\rho^-$ is a weak solution to Eq.\eqref{pme_minus} $\rho^-$ is a subsolution of Eq.\eqref{pme_v} in the weak sense.  By Theorem 5.7 of \cite{v}, we have comparison between weak solutions, which immediate implies $\rho^- \leq v$ a.e.. 

First note that $w$ also satisfies Eq.\eqref{tilde_rho_pde} with $\Phi \equiv 0$, therefore the inequality Eq.\eqref{rho_rho-_rho+} also hold for $w$, namely
\begin{equation}
w - \rho^- \geq -C_1 a,
\end{equation}
where 
$C_1 $ depends on $m, d, M_1, M_2$.

We define $f:= w-\rho^-$, and our goal is to obtain an upper bound for $f$.  More precisely, we want to show there exists some constant $C$ and $\beta$ depending on $m,d,M_1,M_2$, such that $f(x,t) \leq Ca^\beta$ in $\tilde \Omega \times [1,2]$.

 Our strategy is as following.  First, we claim that 
\begin{equation}\label{int_f}
g(T):=\sup_{y\in\tilde \Omega }\int_{B(y,1)\cap \tilde \Omega} f(x,T) dx < C_0 a  \text{ for all }T\in[0,2],
\end{equation}
 where $C_0$ depends on $m, d, M_1, M_2$. We will prove this claim momentarily.  Once we have the claim, we know the space integral of $f(x,t)$ in any unit ball is of order $a$, for $0<t<2$.  To get $f(x,t)\leq O(a^\beta)$ for $t\in[1,2]$, it suffices to show $f$ is H\"older continuous in space
with exponent and constant that are uniform in time  
for all $t\in[1,2]$, which is indeed true, since Theorem 11.2 of \cite{dgv} guarantees this uniform 
H\"older continuity of $\rho^-$ and $w$ for $t\in[1,2]$.  

Now it suffices to prove our claim.   It is proved by writing both equations in weak form, choosing an appropriate test function and applying the Gronwall inequality.  By writing both Eq.\eqref{pme_minus} and Eq.\eqref{pme_v} in weak form and subtracting the two equations, we arrive at
\begin{equation}
\underbrace{\int_{\tilde \Omega} f(x,T)\varphi(x)dx}_{I_1}  =  \underbrace{\int_{\tilde \Omega} f(x,0)\varphi(x)dx}_{I_2} + \int_0^T \underbrace{\int_{\tilde \Omega} \big( w^m - \rho^-|\rho^-|^{m^- - 1}\big)\Delta \varphi(x) + Ma\varphi(x)~ dx}_{I_3}dt, \label{temptemp}
\end{equation} 
where $\varphi\in  C_0^\infty(\tilde \Omega )$ is a test function chosen as follows. For a fixed $T>0$, there exists $z\in \tilde \Omega $, such that the maximum of $\int_{B(y,1)\cap \tilde \Omega} f(x,T) dx$ is achieved at $z$.  We then define $$\varphi(x):= \mu*h^z(x),$$ where $\mu$ is a standard mollifier supported in $B(0,\frac{1}{10})$, and 
\begin{equation}
h^z(x) := \begin{cases} 1-|x-z|^2/2 & \text{for } |x-z|\leq1\\
(|x-z|-2)^2/2 & \text{for } 1<|x-z|\leq 2\\
0 & \text{for } |x-z|>2
\end{cases}
\end{equation}
For such $\varphi$, we have $0<\varphi<1$,
inside the ball $B(z,1)$
 and $\int_{\tilde\Omega} \varphi dx < |B(z,3)| < 6^d $.

To estimate $I_1$, note that $\varphi(x)\geq 1/3$ in $B(z,1)$, and $f(x,T) +C_1 a\geq 0$ in $\tilde \Omega$, which implies
\begin{eqnarray*}
I_1 &=& \int_{\tilde\Omega} (f(x,T)+C_1 a) \varphi(x) dx - \int_{\tilde\Omega} C_1 a \varphi(x) dx\\
&\geq& \frac{1}{3} \int_{B(z,1)\cap \tilde \Omega} (f(x,T)+C_1 a) dx - 6^d C_1 a\\
&\geq& \frac{g(T)}{3} - 6^d C_1 a.
\end{eqnarray*}

For $I_2$, since $f(x,0) = (\frac{m}{m^-})^{m^-} \tilde\rho(x,0)^{1-ca} - \tilde \rho(x,0),$
we would obtain $f(x,0)<C_2 a$, where $C_2$ depends on $m,\|\tilde \rho(\cdot, 0)\|_\infty$ and $c$, (hence depends on $m,d,M_1, M_2$), which yields
\begin{equation*}
I_2 \leq C_2 a \int_{\tilde\Omega} \varphi(x) dx \leq 6^d C_2 a.
\end{equation*}

Now we start to estimate $I_3$. Due to the definition of $m^-$ in Eq. \eqref{def_m*}, we have  $m^- - m \leq 2(m-1)ca$. Also,
we can derive some \emph{a priori} bound of $\rho^-(x,t)$ and $w(x,t)$ for $t\in[1,2]$, which depend on $m, d, M_1, M_2$.  Then we have
\begin{equation*} \Big|w^m - \rho^-|\rho^-|^{m^- - 1}  \Big | \leq C_3 |w-\rho^-| + C_4 a \quad \text{in } \tilde\Omega \times [0,2],
 \end{equation*}
 where $C_3, C_4$ depends on $m, d, M_1, M_2$. 
Together with the fact that $|\Delta \varphi|$ is bounded, in particular by $d$,
in $B(z,3)$ and vanishes outside of $B(z,3)$,  we obtain the following bound for $I_3$:
 \begin{eqnarray}
\nonumber I_3 &\leq& \int_{\tilde\Omega} (C_3|f|+ C_4 a)|\Delta \varphi| + Ma\varphi dx\\
\nonumber&\leq& dC_3\int_{B(z,3)\cap\tilde\Omega} |f(x,t)| dx + 6^d(d C_4 + M)a\\
&\leq& d C_3 \sum_{i=1}^{c_d} \int_{B(z+x_i,1)\cap\tilde\Omega} |f(x,t)| dx + 6^d(d C_4 + M)a, \label{estimate_I3}
\end{eqnarray} 
where in the last inequality we denote by $c_d$ the number such that $B(0,3)$ can be covered by $c_d$ numbers of balls of radius 1, centered at $x_1, \ldots x_{c_d}$ respectively.  Note that $c_d$ is a constant only depending on $d$. 

Finally, we wish to control $\int_{B(z+x_i,1)\cap\tilde\Omega} |f| dx$.  Note that $f\geq -C_1a$ implies $|f| \leq f+2C_1a$, which yields
\begin{eqnarray*}
\int_{B(z+x_i,1)\cap\tilde\Omega} |f(x,t)| dx &\leq& \int_{B(z+x_i,1)\cap\tilde\Omega} f dx + 2^d 2C_1a\\
&\leq& g(t) + 2^{d+1}C_1a
\end{eqnarray*}
Plugging it into Eq.\eqref{estimate_I3}, we obtain
$$I_3 \leq d C_3 c_d g(t) + (d C_3 c_d 2^{d+1}C_1 + 6^d d C_4 + 6^dM) a$$

By putting estimates of $I_1, I_2, I_3$ together, we have
\begin{equation*}
g(T) \leq C_5\int_0^T g(t)dt + C_6 a \quad\text{for }T\in[0,2]
\end{equation*}
where $C_5, C_6$ only depend on $m,d, M_1, M_2$. And for initial data, we have $g(0) \leq |B(0,1)| \sup_x f(x,0) \leq 2^d C_2 a$. By Gronwall inequality, we have $g(T) \leq C_0a$ for all $T\in[0,2]$, where $C_0$ only depends on $m,d,M_1, M_2$, hence our claim Eq.\eqref{int_f} is proved.
\end{proof}

\section*{Acknowledgements}
L.Chayes was partially supported by NSF grant DMS-0805486, I.Kim was partially supported by NSF grant DMS-0970072, and Y.Yao was partially supported by NSF grant DMS-0805486 and DMS-0970072.


\begin{thebibliography}{[MTTTT]}

\bibitem[BRB]{brb} J. Bedrossian, N. Rodriguez and A. Bertozzi,  Local and global well-posedness for aggregation equations and Patlak-Keller-Segel models with degenerate diffusion, \emph{Nonlinearity}, 24(2011): 1683-1714.


\bibitem[Dib]{dib} E. DiBenedetto,  Continuity of weak solutions to a general porous media equation, \emph{Indiana Math. Univ. J.}, 32(1983): 83-118. 

\bibitem[DGV]{dgv}  E. DiBenedetto, U. Gianazza, V. Vespri, Harnack estimates for quasi-linear degenerate parabolic differential equations. \emph{Acta Math.}, 200(2008), 181-209.


\bibitem[BS]{bs} A. Bertozzi and D. Slep\v{c}ev, Existence and uniqueness of solutions to an aggregation equation with degenerate diffusion, \emph{Commun. Pure Appl. Anal.}, 9(2010): 1617-1637. 

 \bibitem[BH]{hilhorst} M. Bertsch and D. Hilhorst,  A density dependent diffusion equation in population 
dynamics: stabilization to equilibrium, \emph{SIAM J. Math. Anal.}, 17(1986): 863-883. 

\bibitem[BCM]{bcm} Silvia Boi, Vincenzo Capasso and Daniela Morale, 
Modeling the aggregative behavior of ants of the 
species polyergus rufescens, \emph{Nonlinear Anal. Real World Appl.}, 1(2000): 163-176.


\bibitem[BCM2]{bcm2}Martin Burger, Vincenzo Capasso, and Daniela Morale, On an aggregation model with long and short range interactions, \emph{Nonlinear Anal. Real World Appl.}, 8(2007): 939-958.

\bibitem[BF]{bf}M. Burger and M. D. Francesco, Large time behavior of nonlocal aggregation models with 
nonlinear diffusion, \emph{Netw. Heterog. Media}, 3(2008): 749-785.




\bibitem[CJMTU]{cjm} J. A. Carrillo,  A. J\"ungel,  P.A. Markowich, G. Toscani, A. Unterreiter, Entropy dissipation methods for degenerate parabolic problems and generalized Sobolev inequalities, \emph{Monatsh. Math.}, 133 (2001): 1-82.

\bibitem[CP]{cp} L. Chayes and V. Panferov, The McKean-Vlasov equation in finite volume, 
\emph{J. Statist. Phys.}, 138(2010): 351-380.

 \bibitem[HV]{hv} M. A. Herrero and J. L. Velazquez,  Chemotactic collapse for the Keller-Segel model,\emph{ J. Math. Biol.,} 35(1996):177-194.


\bibitem[KS]{ks} G. Karch and K. Suzuki, Blow-up versus global existence of solutions to aggregation equations, \emph{Appl. Math. (Warsaw)}, 38(2011): 243-258. 

\bibitem[K]{k} I. C. Kim,  Erratum: ``Degenerate diffusion with a drift potential: a viscosity solutions approach'', \emph{Discrete Contin. Dyn. Syst.}, 30(2011): 375-377.

\bibitem[KL]{kl}I. C. Kim and H. K. Lei, Degenerate diffusion with a drift potential: A viscosity
solutions approach, \emph{Discrete Contin. Dyn. Syst.}, 27(2010): 767-786.


\bibitem[KY]{ky} I. Kim and Y. Yao, The Patlak-Keller-Segel Model and Its Variations: Properties of Solutions via Maximum Principle, \emph{SIAM J. Math. Anal.}, 44(2012): 568-602.


\bibitem[LSU]{lsu} O. A. Ladyzenskaja, V. A. Solonnikov, N. N. UralÕceva. \emph{Linear and QuasiÐlinear 
Equations of Parabolic Type}.  AMS, Providence, RI (1968). 

\bibitem[O]{otto} F Otto, The geometry of dissipative evolution equations: the porous medium equation, \emph{Comm. PDE.}, 26(2001): 101-174.

\bibitem[S]{s} Y. Sugiyama, Global existence in sub-critical cases and finite time blow-up in super-critical cases to
degenerate Keller--Segel systems, \emph{Diff. Int. Eqns.}, 19(2006):841-876.

\bibitem[TBL]{tbl} Chad M. Topaz, Andrea L. Bertozzi, and Mark A. Lewis,  A nonlocal continuum model for biological
aggregation, \emph{Bull. Math. Biol.}, 68(2006): 1601-1623.

\bibitem[V]{v} J. Vazquez, \emph{The Porous Medium Equation: Mathematical Theory}, Oxford University Press,  2007.

\end{thebibliography}
\end{document}